\documentclass{amsart}

\usepackage[utf8]{inputenc}
\usepackage{url}
\usepackage{hyperref}
\usepackage[utf8]{inputenc}
\usepackage[margin=1in]{geometry} 
\usepackage{amsmath,amsthm,amssymb, graphicx, multicol, array}
 \usepackage{mathtools}
 \usepackage{mathrsfs}
\usepackage{tikz-cd}
\usepackage{yfonts}
\usepackage{qtree}
\usepackage{ dsfont }
\usepackage{ upgreek }
\usepackage{hyperref}
\usepackage{enumitem}
\usepackage{tikz}
\usepackage{float}
\usepackage{graphicx}
\usepackage[backend=biber,
style=alphabetic,]{biblatex}
\newcommand{\Z}{\mathbb{Z}}

\newcommand{\Q}{\mathbb{Q}}
\newcommand{\C}{\mathbb{C}}

\newcommand{\G}{\mathbb{G}}
\newcommand{\Pj}{\mathbb{P}}

\newcommand{\mc}{\mathcal}

\newcommand{\sslash}{\mathbin{/\mkern-6mu/}}

\newcommand{\proj}{\text{Proj}}

\newcommand{\diag}{\text{diag}}

\numberwithin{equation}{section}
\newtheorem{prop}{Proposition}[section]
\newtheorem{thm}[prop]{Theorem}
\newtheorem{lemma}[prop]{Lemma}
\newtheorem{cor}[prop]{Corollary}

\usepackage[utf8]{inputenc}

\theoremstyle{definition}
\newtheorem{defn}[prop]{Definition}

\newtheorem{rmk}[prop]{Remark}

\title{
Compact Moduli Spaces of Marked Cubic Plane Curves
}
\author{Aaron Goodwin}
\address{
{\small Department of Mathematics,
University of California, Riverside,
900 University Ave.
Riverside, CA 92521
Skye Hall}
}
\email{agood032@ucr.edu }

\begin{document}

\begin{abstract}
We study compactifications of the moduli space of plane cubic curves marked by \(n\) labeled points up to projective equivalence via Geometric Invariant Theory (GIT). Specifically, we provide a complete description of the GIT walls and show that the moduli-theoretic wall-crossing can be understood through analysis of the singularities of the plane curves and the position of the points.
\end{abstract}

\maketitle

\section{Introduction}\label{sec:intro}
Compactifications of the moduli space of smooth curves of genus $g \geq 0$ with $n$ distinct marked points are a central theme in algebraic geometry. We know nowadays that there are many such compactifications whose boundary are controlled by a variety of combinatorial structures 
\cite{hassett2002modulispacesweightedpointed},\cite{smyth2009classificationmodularcompactificationsmoduli}, \cite{Bozlee_2023}, \cite{pandharipande1995geometricinvarianttheorycompactification},  \cite{Giansiracusa_2010}, 
\cite{Jen13}, \cite{giansiracusa2016gitcompactificationsm0nflips}. For example, the boundary of the Deligne-Mumford compactifications $\overline{M}_{g,n}$ are understood via dual graphs and tropical geometry \cite{caporaso2018recursivecombinatorialaspectscompactified} whereas the Hassett compactifications $\overline{M}_{g, \mathbf{w}}$ use weights to determine which points can collide in the boundary \cite{hassett2002modulispacesweightedpointed}. There have been recent efforts to understand all modular compactifications, in the sense of  \cite{smyth2009classificationmodularcompactificationsmoduli}, and to describe the combinatorial gadgets that govern their boundary \cite{Bozlee_2023}.

In this work, we focus on plane curves of degree $d \geq 0$ marked by $n$ weighted points. Our work follows the previous work of Giansiracusa and Simpson for genus zero \cite{Giansiracusa_2010} and Jensen for the case of one point \cite{Jen13}. 
We describe our set up next: For simplicity we work over $\C$, although we expect essentially all of the results to hold over any algebraically closed field. Let  $\mc{C}_{n,d}$ be the parameter space whose points correspond to tuples $(C, p_1, \dots , p_n)$ of a degree $d$ curve $C \subset \Pj^2$ and $n$ closed points of $C$. When we restrict to the locus $\mc{C}_{n,d}^{gen}$ of smooth curves marked with generic points, then there is a quasi-projective variety $\mc{C}_{n,d}^{gen} \slash SL(3)$ whose points correspond to isomorphism classes of pairs $(C, p_1, \ldots p_n)$ (Lemma \ref{lemma:semistable_generic}). We use variation of GIT (VGIT) quotients to construct compactifications of $\mc{C}_{n,d}^{gen} \slash SL(3)$, which we denote by $\overline{M}_{g(d), \mathbf{w}}^{git},$ where $g(d) = \frac{(d-1)(d-2)}{2}$ is the genus of a degree $d$ plane curve. The construction depends on a line bundle $L$ obtained from the Hilbert scheme of curves and the $n$ points, as described in Section \ref{subsec:computing}. 
We construct such line bundles $L$ from a tuple of non-negative integers $(\gamma, w_1, \dots, w_n) \in \mathbb{Z}^{n+1}_{\geq 0}$. As in similar settings, we will see that $\gamma$ can be interpreted as a weight on the curve and the $w_i$ as weights on each point $p_i$. Our first theorem describes the line bundles that give rise to a well-defined coarse moduli space of marked plane curves. The \textit{$SL(3)-$ample cone} is the set of classes of $SL(3)$-linearized line bundles $\mathbf{w} := [L]$ in $NS^{SL(3)}(\mc{C}_{n,d})_\Q$ such that the locus of $L$ semi-stable marked curves is non-empty, $\mc{C}_{n,d}^{ss}(L) \neq \emptyset$  (See Definition \ref{def:stability}). We study $\Lambda(\mc{C}_{n,d})$, the subset of the $SL(3)$-ample cone which lies in the $\gamma w_1 \dots  w_n$ hyperplane, consisting of line bundles generated by $\gamma, w_1, \dots , w_n$ as in Section \ref{subsec:computing}. In Lemma \ref{lemma:Lambda} we show that $\Lambda(\mc{C}_{n,d})$ can be identified with the $SL(3)$-ample cone.

\newtheorem*{thm:ample_cone}{Theorem \ref{thm:ample_cone}}
\begin{thm:ample_cone}
  The cone $\Lambda(\mc{C}_{n,d})$ is 

\begin{align*}
 \bigg\{
(\gamma, w_1, \ldots, w_n) \in \Q^{n+1}
\; \bigg| \;
   w_i \leq \frac{W+\gamma(2d-3)}{3} , \ w_i \leq \frac{W+ \gamma(d-2)}{2}, \ w_i+w_j \leq \frac{2W+ \gamma d}{3}, \ 
   0 < w_i, \ 0 < \gamma
 \bigg\},
\end{align*}
where $W = \sum_1^n w_i$ denotes the total weight of the points.
\end{thm:ample_cone}

Our results generalize \cite[Proposition 4.2]{Jen13}, which solves the problem for a point and arbitrary degree, and \cite[Theorem 1.1]{Giansiracusa_2010}, which solves the case of $d=2$ and arbitrary $n$.
We prove Theorem \ref{thm:ample_cone} in Section \ref{subsec:outerwalls}.

For our second result, we focus on degree $d=3$ and classify all GIT quotients for marked cubic curves. This is possible because we can list all degenerations of a plane cubic, whereas a similar result for arbitrary degree is out of reach due to the lack of such a classification. 
To describe our theorem,  we recall the wall and chamber decomposition of the $SL(3)$-ample cone (\cite{DH98}, \cite{Tha96}): The $SL(3)-$ample cone is divided  by codimension $1$, locally polyhedral \textit{GIT walls} into a finite set of convex \textit{chambers} such that the semistable loci $\mc{C}_{n,d}^{ss}(L)$ and $\mc{C}_{n,d}^{ss}(L')$ are equal if and only if $L$ and $L'$ lie in the same chamber \cite[Theorem 2.3]{Tha96}. To each wall we associate a closed $SL(3)$ orbit which is strictly semi-stable on the wall. The plane curves with closed orbit associated to each wall have positive dimensional stabilizer \cite[Definition 3.3.1]{DH98}. In the cubic case $d=3,$ such curves belong to a relatively small list; see Lemma \ref{lemma:CubicPositStbz}. 
The proof of Theorem \ref{thm:GIT_walls} is given in Section \ref{sec:innerwalls}. We illustrate the chamber decomposition for $n=2$ in Figure \ref{fig:C_{2,3}}.

\newtheorem*{thm:GIT_walls}{Theorem \ref{thm:GIT_walls}}
\begin{thm:GIT_walls}
    For degree $d=3$, there are four types of inner walls of $\Lambda(\mc{C}_{n,3})$, described below. The walls are segments of hyperplanes defined in Table \ref{table:GITwalls}. If the hyperplane segment intersects the interior of $\Lambda(\mc{C}_{n,3})$ then it is a GIT wall and all inner GIT walls are of this form.
    \begin{itemize}
        \item[(i)] For each nonempty, proper subset $I \subset [n]$ there is a hyperplane segment $W(3A_1, I)$ associated to the union of three non-concurrent lines $C(3A_1)$ with points $p_i, i \in I $ supported at a node.  
        \item[(ii)] For each proper subset $I \subset [n]$ there is a hyperplane segment $W(A_2,I)$ associated to the cuspidal curve $C(A_2)$ with points $p_i, i \in I $ supported at the cusp.
        \item[(iii)]  For each ordered pair of disjoint subsets $I,J \subset [n]$ there is a hyperplane segment $W(A_3, I, J)$ associated to  the union of a conic with a tangent line $C(A_3)$ with points $p_i, i \in I $ supported at the tacnode, and $p_j, j \in J $ supported at the unique linear component of the curve.
        \item[(iv)] For each subset $I \subset [n]$ with $|I| \leq n-3$ there is a hyperplane segment $W(D_4, I)$ associated to the cone over three points $C(D_4)$ with points $p_i, i \in I $ supported at the singularity. 
    \end{itemize}
\end{thm:GIT_walls}

\begin{table}[!htbp]
\renewcommand{\arraystretch}{1.5}
 \begin{tabular}{|p{2cm}| p{5cm}|p{8cm}|}
\hline 
Hyperplane Segment & Hyperplane &  Boundary Conditions \\
\hline 

$W(3A_1, I)$ & $ \sum_{i \in I} w_i - \frac{1}{2} \sum_{j \notin I} w_j = 0$ & $w_m \leq  \frac{1}{2} \sum_{j \notin I} w_j + \gamma, \ \text{for each} \ m \notin I.$ \\

\hline
$W(A_2, I)$ &  $  \sum_{i \in I} w_i  - \frac{4}{5} \sum_{j \notin I} w_j  + \frac{3}{5}\gamma = 0$  & $\sum_{i \in I} w_i \leq \frac{1}{2} \sum_{j \notin I} w_j$ . \\

\hline
$W(A_3, I, J)$ &  $ \sum_{i \in I} w_i -  \sum_{k \notin I \cup J} w_k + \gamma = 0$  & $ \sum_{i \in I} w_i - \gamma \leq \sum_{j \in J} w_j \leq \sum_{i \in I} w_i + 2 \gamma$ . \\

\hline
$W(D_4, I)$ &  $ \sum_{i \in I} w_i - \frac{1}{2} \sum_{j \notin I} w_j + \frac{3}{2} \gamma = 0$ &  There exists a partition $B_1 \sqcup B_2 \sqcup B_3$ of $\{1, \dots , n\} \ \backslash \  I$ such that $  \sum_{j \in B_k} w_j  \leq \sum_{i \in I} w_i + 2 \gamma $ for each $ k \in \{1,2,3\}$ .   \\

\hline
\end{tabular}
\vspace{0.5cm}
\caption{
The walls in Theorem \ref{thm:GIT_walls} are given by intersecting the vanishing locus of a hyperplane with a set of half-planes that cut out the boundary of the wall. The above hyperplane segments are GIT walls if they intersect $\Lambda(\mc{C}_{n,3})$.}
\label{table:GITwalls}
\end{table}

We illustrate Theorem \ref{thm:GIT_walls} for plane cubics with two marked points in Corollary \ref{cor:2pts} and Figure \ref{fig:C_{2,3}}.

\begin{figure}[htbp]
    \centering
    \includegraphics[scale = 0.55]{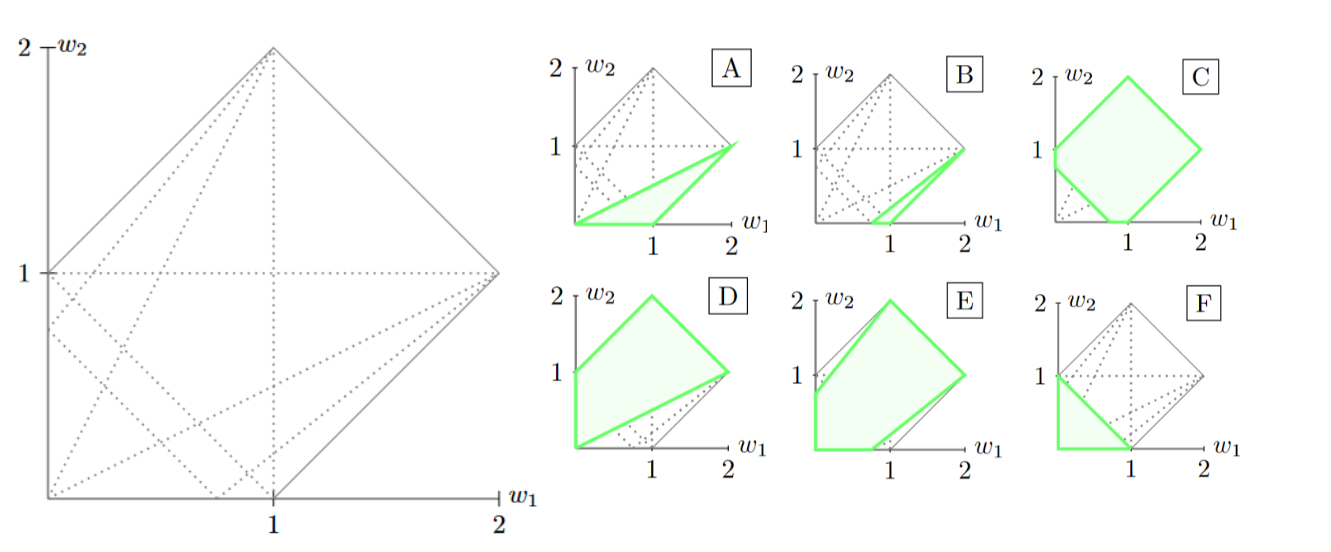} 
    \caption{(Left) The linearization polytope $\Delta(\mc{C}_{2,3})$, which is an intersection of $\Lambda(\mc{C}_{2,3})$ with the hyperplane $\{\gamma=1\},$ as in Section \ref{sec:2pts}. (Right) The locus of line bundles at which marked curves $(C,p_1,p_2)$ with the following pathologies are stable: A) $p_2$ supported at a node. B) $p_2$ supported at a cusp. C) $C$ is cuspidal. D) $p_1$ is supported on a linear component of $C$. E) $p_1$ and $p_2$ are both supported at inflection points. F) $p_1$ and $p_2$ collide. }
    \label{fig:C_{2,3}}
\end{figure}


Next, we discuss the wall-crossing phenomenon for our problem. We recall that given an inner wall, there exists a linearization $L_0$ at the wall, linearizations $L_{+}$ and $L_{-}$ at each chambers adjacent to the wall, and configurations of marked curves that are stable for $L_{+}$, strictly semistable for $L_0$, and unstable for $L_{-}$.  We describe the marked curves that transition from stable to unstable when crossing each GIT wall from Theorem  \ref{thm:GIT_walls}. For this purpose, we recall that each of our walls is associated to a curve $C(T)$ and one (or two) subsets of $[n] := \{1, \dots , n\}$. Let $S(T, I, +)$ be the general curve that satisfies two conditions: it is stable when the value of the linear function defining $W(T,I)$ (Table \ref{table:GITwalls}) is $0< \epsilon \ll 1$ and the boundary conditions are satisfied, and it is unstable whenever the value of the linear function is negative. Conversely, $S(T, I, -)$ denotes the general curve that is stable when the value of the linear function defining $W(T,I)$ is $-1 \ll \epsilon < 0$ and the boundary conditions are satisfied, and is unstable whenever the value of the linear function is positive. For line bundles $L_0$ at the wall, the GIT quotient $\overline{M}^{git}_{g(d),[L_0]}$ identifies the $SL(3)$ orbits of strictly semi-stable marked curves whose closures intersect. Each of these strictly semi-stable curves contains a unique closed orbit in its orbit closure. These closed orbits correspond to the curves $C(T)$ described above, with points $p_i$ such that the marked curve $S(T,I,0) := (C(T), p_1, \dots, p_n)$ has positive dimensional stabilizer. $S(T,I,0)$, is a common degeneration of $S(T,I,-)$ and $S(T,I,+)$.


\newtheorem*{thm:wallcrossing}{Theorem \ref{thm:wallcrossing}}
\begin{thm:wallcrossing}
With notation as above, the changes of stability and semi-stability for marked curves when crossing the walls from Theorem  \ref{thm:GIT_walls} are described as follow: 

\begin{itemize}
\item[(Case i)] $S(3A_1,I,0)$ is a union of three non-concurrent lines. The marked points indexed by $I$ coincide at the nodal intersection of two lines and the remaining marked points lie on the third line.
\begin{itemize}
    \item[$\bullet$] $S(3A_1,I,-)$ is an irreducible nodal curve with the marked points indexed by $I$ coinciding at the  $A_1$ singularity and the remaining marked points in general position.
    \item[$\bullet$] $S(3A_1,I,+)$ is the union of a conic and a transverse line, with marked points $p_j$ lying in general position on the linear component for $j \notin I.$
\end{itemize}
\item[(Case ii)] $S(T,I,0)$ is a plane cubic with a cuspidal singularity. The marked points indexed by $I$ coincide at the $A_2$ singularity and the others coincide at the curves unique inflection point.
\begin{itemize}
    \item[$\bullet$] $S(A_2,I, -)$ is an irreducible cuspidal cubic curve with marked points indexed by $I$ coinciding at the $A_2$ singularity.
    \item[$\bullet$] $S(A_2,I,+)$ is a smooth cubic curve with marked points $p_j$ coinciding at an inflection point for $j \notin I$. 
\end{itemize}
\item[(Case iii)] $S(A_3,I,J,0)$ is the union of a conic and a line which intersect at a tacnode. The marked points indexed by $I$ coincide at the $A_3$ singularity, those indexed by $J$ lie on the linear component of the curve, and the rest lie on a single line in the plane which is also tangent to the conic.
    \begin{itemize}
      \item[$\bullet$] $S(A_3, I, J,-)$ is the union of a conic with a tangent line with marked points indexed by $I$ coinciding at the $A_3$ singularity and those indexed by $J$ lying on the tangent line.
     \item[$\bullet$] $S(A_3,I,J,+)$ is a smooth cubic curve $C$ with marked points $p_k$ coinciding for $k \notin I \cup J$ and the marked points indexed by $J$ coinciding at the transversal intersection of $C$ with the line tangent to $C$ at $p_k$.
    \end{itemize}
    \item[(Case iv)] $S(D_4,I,0)$ is a plane cubic with a $D_4$ singularity. The marked points indexed by $I$ coincide at the $D_4$ singularity. The marked points indexed by $B_1$, $B_2$, and $B_3$ coincide at $3$ points on the $3$ linear component of the curve, respectively, and these $3$ points are collinear.
   \begin{itemize}
    \item[$\bullet$] $S(D_4,I, -)$ is three concurrent lines with marked points indexed by $I$ coinciding at the $D_4$ singularity and the points indexed by $B_1$, $B_2,$ and $B_3$ lying on the three lines, respectively.
    \item[$\bullet$] $S(D_4,I, +)$ is a smooth cubic curve with the points indexed by $B_1$ coinciding at a point $q_1$, the points indexed by $B_2$ coinciding at a different point $q_2$, and the points indexed by $B_3$ coinciding at a third point on the line $\overline{q_1q_2}$.
    \end{itemize}
\end{itemize}
See Figure \ref{fig:WallCrossing} for an illustration of the walls.
\end{thm:wallcrossing}

\begin{figure}
    \centering
    \includegraphics[scale = 0.45]{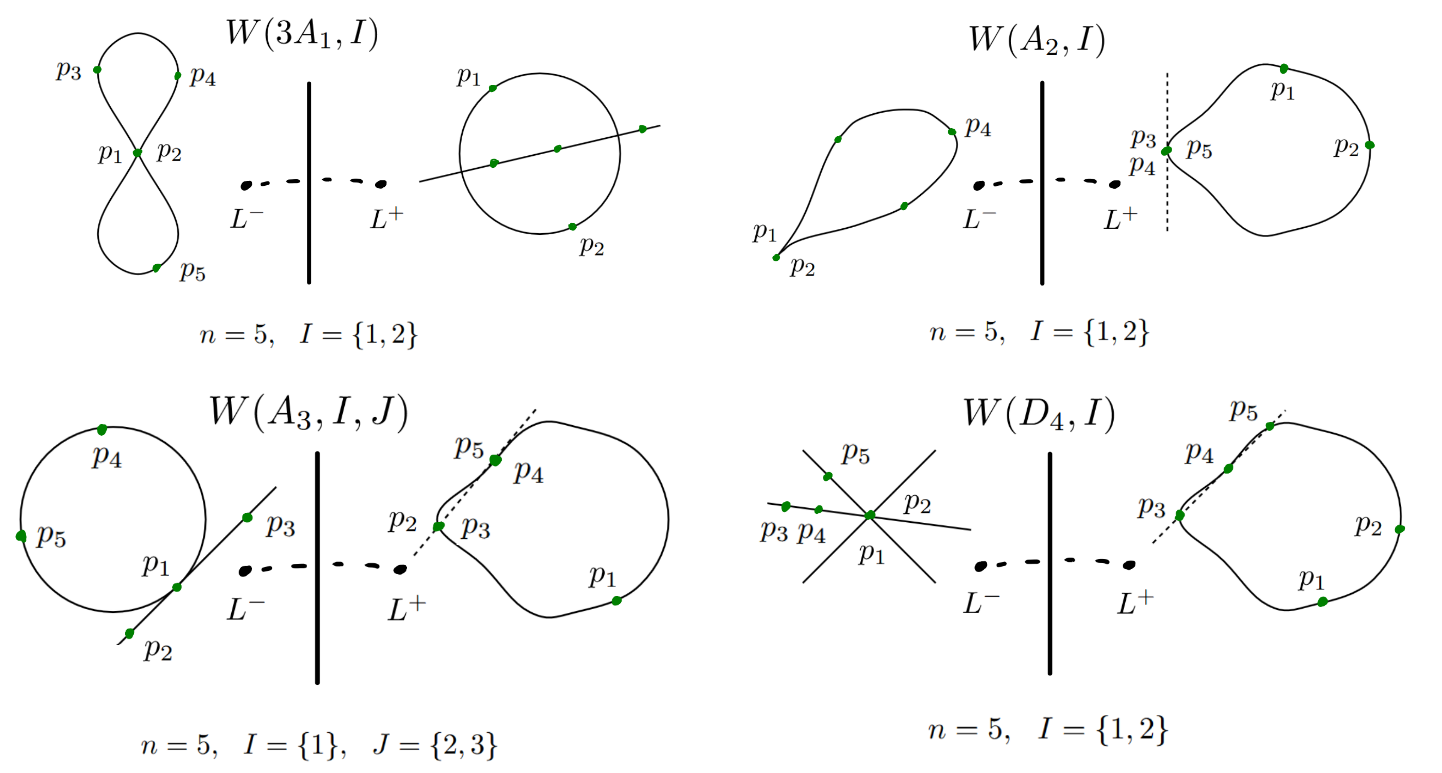}
    \caption{Illustration of the wall crossing behavior found in Theorem \ref{thm:wallcrossing}. The curves $S(T,I, \pm)$ are stable in the chamber containing $L^{\pm}$ and unstable in the chamber containing $L^{\mp}$. }
    \label{fig:WallCrossing}
\end{figure}
In Section \ref{sec:app} we consider applications of our results in how our compactifications relate to other moduli spaces.  In his thesis \cite{Laz06}, Radu Laza constructs a series of compactifications of the moduli space of \textit{degree $3$ pairs} consisting of a cubic curve and a line in $\Pj^2$  by taking the VGIT quotients of $\Pj(\Gamma(\Pj^2, \mc{O}(3))) \times \Pj(\Gamma(\Pj^2, \mc{O}(1))) $ by $SL(3)$, denoted $M^{1,3}_{pairs}(t)$ . He finds a VGIT chamber corresponding to the line bundle parameter $t= \frac{3}{2}- \epsilon$ such that $M^{1,3}_{pairs}(\frac{3}{2}- \epsilon)$ is the moduli space of pairs $(C,L)$ where $C$ has at worst isolated singularities of type $A_k$ and $L$ is a line intersecting $C$ transversely. Laza then considers the moduli spaces $M^{(1,3) lab}_{pairs}(t)$ of pairs of a plane cubic and a line, with labeled intersection. He uses a classical construction to prove that $M^{(1,3) lab}_{pairs}(\frac{3}{2}- \epsilon)$ is the coarse moduli space for cubic surfaces containing a marked Eckardt point and at worst $A_k$ singularities.
In our case, there is a chamber represented by $\mathbf{w}$ giving rise to a compact moduli space of plane cubics with two marked points $\overline{M}^{git}_{1, \mathbf{w}}$ and a map from $\overline{M}^{git}_{1, \mathbf{w}}$ to $M^{(1,3) lab}_{pairs}(\frac{3}{2}- \epsilon)$  given by taking $(C,p_1, p_2)$ to $(C, \overline{p_1p_2})$ and marking the intersection points of $C \cap \overline{p_1p_2}$ with respect to the ordering of $p_1$ and $p_2$.

\newtheorem*{cor:FromM2toRadu}{Corollary \ref{cor:FromM2toRadu}}
\begin{cor:FromM2toRadu}
    There exists an isomorphism
    \[
       \phi^{lab}:\mc{C}_{2,3}^s \slash_{L} SL(3) \rightarrow M^{(1,3)lab}_{pairs}\left(
       \frac{3}{2}- \epsilon
       \right) 
       \cong 
       \mathbb{P}(1,2,2,3)
    \]
    where $L$ is a line bundle corresponding to a vector $\mathbf{w}$ in the GIT chamber $\{ w_1 > 1, w_2 > 1, w_1 + w_2 < 3\} \subset \Delta(\mc{C}_{2,3})$.
\end{cor:FromM2toRadu}

Next, we discuss the relation between our compactifications of marked plane cubics and the moduli space of marked elliptic curves.  Recall that an elliptic curve is a curve of genus $1$ with a marked point \cite{Dol12}. If $(E,p)$ is an elliptic curve then the linear series $|3p|$ embeds $E$ in $\Pj^2$ as a cubic curve such that $p$ is an inflection point. For this reason, we are interested in marked cubics for which one of the marked points is specifically an inflection point. 
In Section \ref{sec:planecubicsM1n},  we construct a space $\mc{C}_{n,3}'$ which parametrizes tuples consisting of a plane cubic, $n$ marked points, and an $(n+1)^{th}$ marked inflection point. Forgetting the last $(n+1)^{th}$ point gives a $9:1$ cover $\mc{C}_{n,3}' \rightarrow \mc{C}_{n,3}$. We construct GIT quotients $\overline{M}^{git}_{1,\mathbf{w}+} := \mc{C}_{n,3}' \sslash_L SL(3)$ and show that they compactify $M_{1,n+1}$.
The use of the single slash for the quotient indicates that the stable and semi-stable loci coincide for $L$, and the GIT quotient is therefore a geometric quotient.

\newtheorem*{thm:M_1n}{Corollary \ref{thm:M_1n}}
\begin{thm:M_1n}
    $M_{1,n+1}$ is isomorphic to an open subset of $\mc{C}_{n,3}^{'s} \slash_L SL(3)$ for some $SL(3)-$linearized line bundle $L$.
\end{thm:M_1n}
To study the birational map from $\overline{M}_{1,n}$ to 
$\overline{M}^{git}_{1,\mathbf{w}+}$ 
is the focus of the author's ongoing work.

\subsection*{Acknowledgements}
The author would like to thank their advisor Patricio Gallardo for suggesting this problem and for his thoughtful guidance and support throughout. The author would also like to thank Jos\'e Gonz\'alez for his advice and his enthusiastic teaching. This work is partially supported by the National Science Foundation under Grant No. DMS-2316749 (PI: P Gallardo).

\tableofcontents

\section{Preliminaries}

\subsection{Geometric Invariant Theory}

Throughout this section $G$ will be reductive group acting algebraically on a projective variety $X$.

\begin{defn} A \textit{categorical quotient} is a variety $Y$ with a morphism $p:X \rightarrow Y$ such that $p$ is $G-$invariant and for any other $G-$invariant map $q:X \rightarrow Z$ there exists a unique map $r: Y \rightarrow Z$ such that $q = r \circ p$. A categorical quotient is called a \textit{geometric quotient} if furthermore:
\begin{enumerate}[label=\arabic*)]
    \item For any open subset $U \subset Y$, the homomorphism $p^{\#}: \mc{O}(U) \rightarrow \mc{O}(p^{-1}(U))$ is an isomorphism onto the subring $\mc{O}(p^{-1}(U))^G$ of $G-$invariant sections.
    \item $p$ is a surjective, open map.
    \item the fibers of $p$ are precisely the $G-$orbits of $X$.
\end{enumerate}
\end{defn}

Let $L$ be a $G-$linearized ample line bundle on $X$. We recall the semi-stable, stable, and unstable loci of $X$ with respect to $L$. Theorem \ref{theorem:GIT_quotients} describes a construction of geometric quotients from the stable locus and compactifications of these quotients from the semi-stable locus.

\begin{defn}\label{def:stability}
    \begin{enumerate}[label=\roman*)]
        \item $x \in X$ is called \textit{semi-stable} with respect to $L$ if there exists an $m>0$ and a $G-$invariant section $s \in \Gamma(X, L^{\otimes m})^G$ such that $s(x) \neq 0$. The set of semi-stable points, denoted $X^{ss}(L)$, is called the \textit{semi-stable locus}.
        \item $x \in X$ is called \textit{stable} with respect to $L$ if it is semi-stable, its $G$-orbit $G \cdot x$ is closed in $X^{ss}(L)$, and it has a finite stabilizer $G_x$. The set of stable points, denoted $X^s (L)$ is called the \textit{stable locus}.
        \item The \textit{unstable locus} $X^{us}(L)$ is the complement $X \ \backslash \  X^{ss}(L)$.
        \item The \textit{strictly semi-stable} locus $X^{sss}(L)$ is the complement $X^{ss}(L) \ \backslash \ X^s(L)$.
    \end{enumerate}
\end{defn}

\begin{thm}\label{theorem:GIT_quotients}
    For a given $G-$linearized line bundle $L$ on $X$, the semi-stable locus $X^{ss}$ is an open subset of $X$. The stable locus $X^s$ is an open subset of $X^{ss}$.
    
    \begin{enumerate}[label=\arabic*)]
        \item There exists a categorical quotient $\pi: X^{ss}(L) \rightarrow X^{ss}(L) \sslash G := \proj \bigoplus_{m\geq 0} \Gamma(X, L^{\otimes m})^G $. This quotient is a projective variety.
        \item $\pi$ restricts to a geometric quotient of the stable locus $X^s (L) \rightarrow X^s (L) \slash G := \pi(X^s(L)) $.
    \end{enumerate}
    
\end{thm}
\begin{proof}
    \cite{Dol03} Chapter 9 Theorem 8.1 and Proposition 8.1.
\end{proof}

The reader will notice that the GIT quotients in Theorem \ref{theorem:GIT_quotients} depend on the choice of a $G$-linearization on $X$. If $G$ has a trivial character group and $\text{Pic}(G)$ is trivial, then each ample line bundle has at most one $G$-linearization, see \cite[Section 7]{Dol03}. However,  the quotients can change when $L$ varies in $\text{Pic}(X)$. Moreover, there is a wall-chamber decomposition associated to the different isomorphism classes of GIT quotients. We recall the relevant facts.

\begin{prop}
    If $L$ is an ample linearization, then the semi-stable locus $X^{ss}(L)$ and the quotient $X^{ss}(L) \sslash G$ depend only on the $G-$algebraic equivalence class of $L$ in $NS^G(X)$.
\end{prop}

\begin{proof}
   See \cite{Tha96} Proposition 2.1.
\end{proof}

For this reason we consider $G-$algebraic equivalence classes of linearized line bundles in the N\'eron Severi group $NS^G(X)$. Tensoring with $\Q$ gives a finite dimensional vector space $NS^G(X)_\Q:= NS^G(X) \otimes \Q$ \cite[Section 2]{Tha96}. Let $E^G$ be the set of equivalence classes of line bundles for which the semi-stable locus is non-empty. The \textit{$G$-ample cone} is defined to be $E^G \otimes \Q \subset NS^G (X)_{\Q}.$ The set of line bundles in $NS^G(X)_{\Q}$ at which a point $x \in X$ is strictly semi-stable is called a \textit{GIT wall.} The $G-$ample cone is divided by the GIT walls into finitely many convex chambers such that the semi-stable locus does not change within each chamber \cite{Tha96}, \cite{DH98}.

Let $L$ be an ample $G$-linearized line bundle on $X$. For some $m \gg 0$, the tensor power $L^{\otimes m}$ embeds $X$ in $\Pj^N$ such that $G$ acts on $X$ as a subgroup of $GL(N+1)$. A \textit{one-parameter subgroup} of an algebraic group $H$ is a morphism of algebraic groups $\G_m \rightarrow H$, where $\G_m$ is the multiplicative group scheme over $k$. Given any one-parameter subgroup $\lambda$ of $G \subset GL(N+1)$ we can find a basis of $\Pj^N$ such that the image of $\lambda$ is diagonal of the form
$$ \lambda: t \mapsto \begin{bmatrix}
    t^{r_1} & & & 0\\
    & t^{r_2} & & \\
    & & \ddots & \\
   0 & & & t^{r_{N+1}}
\end{bmatrix} $$
for $r_i \in \Z$. Let $x \in X$ have homogeneous coordinates $[X_1: \dots : X_{N+1}]$ with respect to the the diagonalizing basis for $\lambda$. Then the \textit{Hilbert-Mumford Numerical Function} identifies the minimal exponent $r_i$ such that $X_i$ is non-zero:
$$\mu^L (x, \lambda) := \min \{r_i \ | \ X_i \neq 0\}.$$
This number is independent of the diagonalizing basis.

\begin{thm}{(The Hilbert-Mumford Numerical Criterion)}\label{theorem:Mumf}
    Let $S$ be the set of all one-parameter subgroups of $G$. Then
    \begin{itemize}
        \item $X^{ss}(L) = \{ x \in X \ | \ \mu^L(x, \lambda) \leq 0 \ for \ all \ \lambda \in S  \}$.
            \item $X^{s}(L) = \{ x \in X \ | \ \mu^L(x, \lambda) < 0 \ for \ all \ \lambda \in S  \}$.
    \end{itemize}
\end{thm}

For a proof of Theorem \ref{theorem:Mumf} and further details see \cite[Chapter 2.1]{Mum82}. Caution: Mumford uses the opposite $\pm$ signature for $\mu^L(x, \lambda);$ both conventions occur in the literature. 

\subsection{The Numerical Criterion for marked plane curves}
Our moduli spaces $\overline{M}^{git}_{g(d), \mathbf{w}}$ are formed as the GIT quotients of a parameter space $\mc{C}_{n,d}$ of degree $d$ plane curves with $n$ marked points by the natural action of $SL(3)$ acting as linear automorphisms of the projective plane. Our parameter spaces $\mc{C}_{n,d}$ are built from the product of  $n$ copies of $\Pj^2$ with $\Pj(\Gamma(\Pj^2, \mc{O}(d)))$, the Hilbert scheme of plane curves of degree $d$. The vector space $\Gamma(\Pj^2, \mc{O}(d))$ is generated by the $\binom{d+2}{2}$ degree $d$ monomials $x^iy^jz^k$, where $x,y,$ and $z$ are generators of $\Gamma(\Pj^2, \mc{O}(1))$. We specify these generators in the following paragraph. Using coordinates $a_{ijk}$ for $(\Gamma(\Pj^2, \mc{O}(d))$ and $(x_r:y_r:z_r)$ for the $r^{th}$ copy of $\Pj^2,$ we define $\mc{C}_{n,d}$ as the subvariety of $\Pj(\Gamma(\Pj^2, \mc{O}(d))) \times (\Pj^2)^n$ cut out by the $n$ \textit{incidence relations}: $$\sum_{i+j+k=d} a_{i,j,k}x_r^iy_r^jz_r^k = 0, \ \ r = 1, \dots , n.$$
The parameter space $\mc{C}_{n,d}$ is therefore the variety $\mc{U}^n$ studied by Caminata, Moon, and Schaffler, which is normal and Gorenstein \cite[Remark 4.14]{CamMoonSchaff}.

$SL(3)$ acts naturally on $\Pj^2$ as linear automorphisms. It acts by inverses on the space of functions $\Gamma(\Pj^2, \mc{O}(d)),$ that is $g \cdot f(p) := f(g^{-1} \cdot p)$. Then the product of these actions of $SL(3)$ on $\Pj(\Gamma(\Pj^2, \mc{O}(d))) \times (\Pj^2)^n$ restricts to an action on $\mc{C}_{n,d}$: if the points $p_i$ lie on a curve $C$ then $g \cdot p_i$ will lie on $g \cdot C.$ We say two marked curves, $(C, p_1, \dots p_n)$ and $(D, q_1, \dots , q_n)$ are \textit{linearly equivalent} if they belong to the same $SL(3)$ orbit, namely if there exists a $g \in SL(3)$ such that $g \cdot C=D$ and $g \cdot p_i=q_i$ for each $i \in \{1, \dots , n\}.$

We now describe a modification of the Hilbert-Mumford Numerical Criterion which gives an effective way of computing the GIT stability for marked plane curves. We fix a maximal torus $T$ of $SL(3)$ and choose coordinates $x, y, z$ on $\Pj^2$ such that $T$ is diagonal. A one-parameter subgroup $\lambda$ of $T$ has the form
$$\lambda: t \mapsto \begin{bmatrix}
    t^a & & 0 \\
    & t^b & \\
    0 & & t^c
\end{bmatrix}$$
and will be denoted by $\diag\{a,b,c\}$, with the exponents $a,b,$ and $c$ the \textit{weights} of $\lambda$. Let $N \cong \Z^2$ be the lattice of one-parameter subgroups of $T$. GIT stability is determined by the sign of the Hilbert-Mumford function, which is invariant under scaling the weights of $\lambda$ by a positive factor. We therefore work with the vector space of rational one-parameter subgroups $N \otimes \Q$. If $a=1 \geq b \geq c$ then $\lambda = \lambda_r := \diag\{1,r,-1-r\}$ for some $r \in [-1/2, 1]$. We call such one-parameter subgroups \textit{normalized}.

In Theorem \ref{theorem:NC} we reduce the Hilbert-Mumford Numerical Criterion to normalized one-parameter subgroups of our fixed maximal torus. This comes at the expense of having to consider the entire $SL(3)$ orbit of marked curves $x\in \mc{C}_{n,d}$, which can be seen as considering a marked curve under all coordinate systems on $\Pj^2$.
\begin{thm}{(Our Numerical Criterion)}\label{theorem:NC}
    Let $S'$ be the set of all normalized one-parameter subgroups of $T$. Then 
    \begin{itemize}
        \item $\mc{C}_{n,d}^{ss}(L) = \{ x \in \mc{C}_{n,d} \ | \ \mu^L(g \cdot x, \lambda_r ) \leq 0 \ for \ all \ \lambda_r \in S', g \in SL(3)  \}$.
            \item $\mc{C}_{n,d}^s(L) = \{ x \in \mc{C}_{n,d} \ | \ \mu^L(g \cdot x, \lambda_r ) < 0 \ for \ all \ \lambda_r \in S', g \in SL(3)  \}$.
    \end{itemize}
\end{thm}

\begin{proof}
    We prove that there exists a normalized one-parameter subgroup $\lambda_r$ of $T$ making $\mu^L(g \cdot x, \lambda_r) \geq 0$ for some $g \in SL(3)$ if and only if there exists a one-parameter subgroup $\lambda$ of $SL(3)$ making $\mu^L(x,\lambda) \geq 0$. The first direction is immediate, taking $g$ to be the identity matrix and $\lambda := \lambda_r$. Now suppose there is a $\lambda$ making $\mu^L(x,\lambda) \geq 0$. The image of $\lambda$ is contained in some maximal torus of $SL(3)$. Since all maximal tori of $SL(3)$ are conjugate, the image of $\lambda$ is conjugate to a subgroup of $T$. That is, there exists a $g \in SL(3)$ such that $g \lambda g^{-1}(\C^*) \subset T$. After a permutation of coordinates we may furthermore assume that the weights of $g \lambda g^{-1}$ satisfy $a \geq b \geq c$.
    Then, using the functorial properties of $\mu$ \cite[Definition 2.2]{Mum82},
    $$\mu^L(g \cdot x, g \lambda g^{-1}) = \mu^L(x, g^{-1}g \lambda g^{-1}g)= \mu^L(x, \lambda) \geq 0.  $$
    By scaling the weights $a,b,c$ of $g \lambda g^{-1}$ by $\frac{1}{a}$ we obtain a normalized one-parameter subgroup satisfying the same inequality. This shows that some $g \cdot x$ in the orbit of $x$ is stable with respect to a normalized one-parameter subgroup of $T$ if and only if $x$ is stable. The proof for semi-stability is identical, with the inequalities replaced by strict inequalities.
\end{proof}

We now introduce combinatorial objects that are useful for computations involving the numerical criterion. Let $\Xi_d$ be the set of degree $d$ monomials in $x,y,$ and $z$. For a point $p \in \Pj^2$ we define the \textit{support} of $p$ to be $\Xi(p) := \{ x_i \in \Xi_1 \ | \ \text{the $x_i$ coordinate of $p$ is nonzero} \}$. For a degree $d$ curve $C$ we define the \textit{support} of $C$ to be $\Xi(C) = \{ m \in \Xi_d \ | \ \text{the coefficient of} \ m  $ $  \text{in the polynomial function defining $C$ is non-zero} \}$. The standard bilinear pairing between characters and one-parameter subgroups of a torus restricts to pairings between elements of $\Xi_d$ and normalized one-parameter subgroups $\lambda_r$ defined by
$$\langle x^a y^b z^c , \lambda_r \rangle := a + rb + (-1-r)c.$$

\begin{defn}\label{defn:MukaiOrder}
The \textit{Mukai Order} on $\Xi_d$ is given by $m \geq m'$ if and only if $\langle m, \lambda_r \rangle \geq \langle m', \lambda_r \rangle$ for all normalized one-parameter subgroups $\lambda_r$.
\end{defn}

\begin{lemma}
 $x^ay^bz^c \geq x^{a'}y^{b'}z^{c'}$ if and only if $a \geq a'$ and $a+ b \geq a' + b'$.
\end{lemma}

\begin{proof}
See \cite{Muk03} chapter 7.2.
\end{proof}

\begin{defn}\label{defn:maxsup}
The \textit{maximal support} of a curve $C$, denoted $\Xi_{max}(C)$, is the set of maximal elements  of the support of $C$ under the Mukai Order. Similarly, we define $\Xi_{min}(p)$ to be the minimal element in the support of $p$. If $S, S' \subset \Xi_d$ are sets of degree $d$ monomials, we say $S \geq S'$ if for each $m' \in S'$ there exists an $m \in S$ with $m \geq m'.$
\end{defn}

\subsection{Computing the stability of a marked curve} \label{subsec:computing}

In this section we describe heuristics for computing the (semi-)stability of a marked curve with respect to a given line bundle. Theorem \ref{thm:generic} demonstrates such computations in detail.

We consider the line bundles on $\mc{C}_{n,d}$ obtained as restrictions of nef line bundles on $\Pj^{\binom{d+2}{d}-1} \times (\Pj^2)^n.$ These are restrictions of line bundles of the form $\pi_0^*(\mc{O}(G)) \times \pi_1^*(\mc{O}(W_1)) \times \dots \times \pi_n^*(\mc{O}(W_n))$ for non-negative integers $G, W_1, \dots , W_n \in \Z$. Since $SL(3)$ has trivial character group there is only one linearization for each line bundle of this form. We will prove that these line bundles span the N\'eron-Severi group $NS^{SL(3)}(\mc{C}_{n,d})_\Q$.

\begin{defn}
    Let $\Lambda(\mc{C}_{n,d})$ be the intersection of the $SL(3)-$ample cone of $\mc{C}_{n,d}$ with the subspace of classes of line bundles generated by the restriction of a line bundles on $\Pj^{\binom{d+2}{d}-1} \times (\Pj^2)^n.$ 
\end{defn}

We denote vectors in this cone by $\mathbf{w} = (\gamma, w_1, \dots , w_n) \in \Q^{n+1}$ and we denote line bundle representatives of the equivalence class $\mathbf{w}$ by $L_{\mathbf{w}}$ or simply $L$ if there is no ambiguity. 

\begin{lemma} \label{lemma:Lambda}
    Let $P:= \Pj (H^0(\Pj^2, \mc{O}(d))) \times (\Pj^2)^n$. There is an isomorphism of \textup{Pic}$(P)_\Q$ with $NS^{SL(3)}(\mc{C}_{n,d})_\Q$ which identifies
    $\Lambda(\mc{C}_{n,d})$ with the $SL(3)$-ample cone.
\end{lemma}

\begin{proof}
    Consider the effective ample divisors  $D_1, \dots , D_n$ on $P$ given by the $n$ incidence relations $D_i := \{ \sum_{i+j+k=d} a_{i,j,k}x_i^iy_i^jz_i^k = 0\}$. Since $\mc{C}_{n,d}$ is the complete intersection of these divisors in $P$, a version of the Lefschetz Hyperplane Theorem \cite[Remark 3.1.32]{LazarsPos} implies that $H^2(P, \Z) \cong H^2(\mc{C}_{n,d}, \Z)$. If we consider the exponential exact sequence of sheaves on $P$ and $\mc{C}_{n,d}$ we get long exact sequences of their cohomology groups. Consider the following piece of the long exact sequences:

\begin{center}

    \begin{tikzcd}
        H^1(P, \mc{O}_P) \ar[r] & H^1(P, \mc{O}_P^*) \ar[r] & H^2(P, \Z) \ar[r] \ar[d, "\simeq"] & H^2(P, \mc{O}_P) \\
        
         H^1(\mc{C}_{n,d}, \mc{O}) \ar[r] & H^1(\mc{C}_{n,d}, \mc{O^*}) \ar[r, "c_1"] & H^2(\mc{C}_{n,d}, \Z) \ar[r] & H^2(\mc{C}_{n,d}, \mc{O}). 
    \end{tikzcd}
    
\end{center}

By the K\"unneth Theorem, $H^i(P, \mc{O}_P)=0$ for $i=1$ and $2$. Therefore the map $H^1(P, \mc{O^*}_P) \rightarrow H^2(P, \Z)$ is an isomorphism $H^1(P, \mc{O^*}_P) \cong H^2(\mc{C}_{n,d}, \mc{O}).$ The map $c_1$ is the first Chern class of a line bundle. It maps algebraic equivalence classes of line bundles in $NS(\mc{C}_{n,d}) $ injectively into $H^2(\mc{C}_{n,d}, \Z) \cong \text{Pic}(P).$ After tensoring with $\Q$ we have an injective map of $\Q$ vector spaces $NS(\mc{C}_{n,d})_\Q \hookrightarrow \text{Pic}(P)_\Q $. On the other hand, the restriction of line bundles map $\text{Pic}(P)_\Q \rightarrow NS(\mc{C}_{n,d})_\Q$ is injective because its image has dimension $n+1$: each of the $n+1$ generators of Pic$(P)_\Q$, $\gamma, w_1, \dots, w_n$, has a distinct effect on the stability of marked curves when they become too large, so they are linearly independent in $NS(\mc{C}_{n,d})$.

Then Pic$(P)_\Q$ and $NS(\mc{C}_{n,d})_\Q$ are equidimensional and the injection $NS(\mc{C}_{n,d})_\Q \hookrightarrow \text{Pic}(P)_\Q $ is an isomorphism. Furthermore, since the character group of $SL(3)$ is trivial, each line bundle on $\mc{C}_{n,d}$ admits a unique $SL(3)$ linearization. So we obtain an isomorphism $NS^{SL(3)}(\mc{C}_{n,d})_\Q \cong \text{Pic}(P)_\Q$.
\end{proof}

Given a vector $\mathbf{w}:=(\gamma, w_1, \dots , w_n)$, let  $M$ be the least common multiple of the denominators of $\gamma, w_1, \dots , w_n$. For each marked curve $(C, p_1, \dots , p_n)$ we study the set of vectors $\mathbf{w}$ in $\Lambda(\mc{C}_{n,d})$ such that $(C, p_1, \dots , p_n)$ is $L_{M\mathbf{w}}$ (semi-)stable.  The line bundle $L_{M\mathbf{w}}$ embeds $\mc{C}_{n,d}$ into $\Pj^N$ by a composition of the Segre embedding with the Veronese embeddings of  respective degrees $M \gamma, M w_1, \dots ,$ and $ M w_n$ on $\mc{C}_{n,d} \subset \Pj^{\binom{d+2}{d}-1} \times (\Pj^2)^n $. Let  $\lambda_r$ be the normalized diagonal one-parameter subgroup $diag\{1,r,-1-r\}$. From the definition of $\mu$ we then compute
\begin{equation} 
    \mu^{L_{M\mathbf{w}}}((C,p_1, \ldots, p_n), \lambda_r)
=
M(\sum_{i=1}^n w_i \min \{ \langle x_i , \lambda_r \rangle \ | \ x_i \in \Xi(p_i) \} - \gamma \max \{ \langle m, \lambda_r \rangle \ | \ m \in \Xi(C) \}). 
\end{equation}

Since our Numerical Criterion (Theorem \ref{theorem:NC}) determines stability by the $\pm$ sign of $\mu^L(\mathbf{x}, \lambda_r)$, we are interested in the numerical function up to a positive multiple. We therefore expand its definition to include vectors with non-integer coordinates:

\begin{equation} \label{eq:mu}
    \mu^{\mathbf{w}}((C,p_1, \ldots, p_n), \lambda_r)
:=
\sum_{i=1}^n w_i \min \{ \langle x_i , \lambda_r \rangle \ | \ x_i \in \Xi(p_i) \} - \gamma \max \{ \langle m, \lambda_r \rangle \ | \ m \in \Xi(C) \}. 
\end{equation}

\begin{rmk}
    We will also write $\mu^L$ for the function $\mu^{\mathbf{w}}$ when $L$ is a line bundle representative of the equivalence class $\mathbf{w}$.
\end{rmk}

By Theorem \ref{theorem:NC}, the marked curve $(C, p_1, \dots , p_n)$ is $L_{\mathbf{w}}$ stable if $\mu^{\mathbf{w}}( g \cdot (C,p_1, \ldots, p_n), \lambda_r)$ is negative for all $g \in SL(3)$ and $r \in [\frac{-1}{2}, 1]$. Notice that the right side of equation \ref{eq:mu} is determined by the monomials in the support of $C$ and  $p_1, \dots , p_n$. We therefore are interested in the \textit{(maximal) support of $g \cdot (C,p_1, \dots, p_n)$}, defined as
$$ \Xi(g \cdot (C, p_1, \dots , p_n)) := (\Xi(g \cdot C), \Xi(g \cdot p_1), \dots , \Xi(g \cdot p_n)) \in \mc{P}(\Xi_d) \times (\mc{P}(\Xi_1))^n$$
$$\text{and} \ \ \Xi_{max}(g \cdot (C, p_1, \dots , p_n)) := (\Xi_{max}(g \cdot C), \Xi_{min}(g \cdot p_1), \dots , \Xi_{min}(g \cdot p_n)) \in \mc{P}(\Xi_d) \times (\mc{P}(\Xi_1))^n$$
where $\mc{P}$ denotes the power set. In fact, because stability is determined by the maximal values of $ \mu^L(g \cdot (C,p_1, \ldots, p_n), \lambda_r)$, we only need to keep track of the $g \in SL(3)$ such that $\mu^L(g \cdot (C, p_1, \dots , p_n), \lambda_r)$ is maximized as a function of $r$ and $L$. We define the following partial order
\begin{defn}\label{defn:prodsupporder}
    For $(X, P_1, \dots , P_n), (Y, Q_1, \dots Q_n) \in \mc{P}(\Xi_d) \times (\mc{P}(\Xi_1))^n$, we say $$(X, P_1, \dots , P_n) \geq (Y, Q_1, \dots Q_n)$$ if $X \leq Y$ and $P_i \geq Q_i$ (as in Definition \ref{defn:maxsup}) for all $i \leq n$.
\end{defn}

Notice the direction of the inequalities in the above definition. This is explained by the following Lemma.
\begin{lemma}\label{lemma:maxmuxi}
    Let $(C, p_1, \dots , p_n)$ and $(D, q_1, \dots , q_n)$ be marked curves. The following are equivalent:
    
    \begin{itemize}
        \item $\Xi((C, p_1, \dots, p_n)) \geq \Xi((D, q_1, \dots, q_n))$ 
        \item $\Xi_{max}((C, p_1, \dots, p_n)) \geq \Xi_{max}((D, q_1, \dots, q_n))$ 
        \item $\mu^L((C, p_1, \dots, p_n), \lambda_r) \geq \mu^L((D, q_1, \dots, q_n), \lambda_r)$ for all $r \in [\frac{-1}{2},1]$ and $L$ in $\Lambda(\mc{C}_{n,d})$.
         
    \end{itemize}
\end{lemma}

\begin{proof}
    This is a direct application of Definition \ref{defn:prodsupporder} to Equation \ref{eq:mu}.
\end{proof}

  Since $\Xi_d$ are finite sets, for a fixed degree $d$ there are only finitely many values of $\Xi(g \cdot (C, p_1, \dots p_n)), g \in SL(3).$ It can be seen (Theorem \ref{thm:generic}, Lemma \ref{lemma:TriangleStbz}) that in practice, finding the maximal values of $\Xi(g \cdot (C, p_1, \dots p_n))$ involves an analysis of the geometry of the marked curve and its intersection with the flag $\{(1:0:0)\} \subset \{z=0\} \subset \Pj^2$. We obtain a finite list of maximal values of  $\mu^L(g \cdot (C, p_1, \dots p_n), \lambda_r),$ considered as functions of $r$ and $L$. These functions are piecewise linear in $r$. Evaluating them at the critical values of $r$ then gives a list $\{F_1, \dots , F_s\}$ of functions of $\mathbf{w} = (\gamma, w_1, \dots , w_n)$, such that $(C, p_1, \dots, p_n)$ is $L_{\mathbf{w}}$ (semi-)stable if and only if $F_j(\mathbf{w})$ is negative (non-positive) for all $j \in \{1, \dots ,s\}$.

In our first example we compute the (semi-)stable locus of a smooth irreducible marked curve with marked points in general position. We will later prove that these curves are semi-stable throughout $\Lambda(\mc{C}_{n,d})$ (Lemma \ref{lemma:semistable_generic}).

Let $(C, p_1 , \dots , p_n)$ be a marked plane curve of degree $d \geq 3$. Suppose $C$ is smooth and the points are general in the following sense:
    \begin{itemize}
        
    \item The $p_i$ are distinct 
    \item No point lies on a flex line of $C$
    \item No three points lie on a line
    \item and no two points lie on a line which is tangent to $C$.
    \end{itemize} 

    Recall that a \textit{flex line} of a plane curve of multiplicity $m$ is a line that intersects the curve at a \textit{flex point} with multiplicity $m>2$.

    \begin{rmk}
        
    The fact that these conditions are general can be seen by induction on $n$. Since a smooth plane curve has finitely many inflection points, the base case is immediate. For the inductive step we show that for a marked curve $(C, p_1, \dots, p_{n-1})$ satisfying these conditions, there is a Zariski open subset of $C$ on which we may pick an $n^{th}$ point $p_n$ such that the conditions are satisfied. There are clearly only finitely many points of $C$ which violate the first two conditions. Only finitely many points violate the third condition because for $p_i, p_j \in \{p_1, \dots , p_{n-1}\}$, the line through $p_i$ and $p_j$ intersects $C$ in at most $d$ points. To satisfy the fourth condition the point $p_n$ must not lie on any line tangent to $p_1, \dots , p_{n-1}$, which precludes finitely many points, and the tangent line to $p_n$ must not intersect any of the $p_1, \dots , p_{n-1}$. However, the number of tangent lines of $C$ which go through $p_i$ can be shown to be finite as a consequence of Bezout's Theorem. It is bounded by the \textit{class} of $C$, which is $d(d-1)$ \cite[Example 10.4.1]{FultonInt}.
    \end{rmk}

\begin{thm}\label{thm:generic}
The general marked curve $(C,p_1, \dots, p_n)$ is $L$ stable if and only if the weights $\gamma, w_1, \dots , w_n$ are all positive and for each distinct pair $i,j \in \{1, \dots , n\}$, $3w_i - \sum_1^k w_k + 3 \gamma -2 \gamma d < 0, \  $ $2w_i - \sum_{k=1}^n w_k - \gamma d + 2 \gamma < 0,  \ $ and $ \ 3w_i + 3w_j - 2 \sum_1^n w_k - \gamma d < 0$.
    
\end{thm}
\begin{proof}
 Our first step is to use the geometry of $(C,p_1, \dots , p_n)$ to determine the maximum possible supports of $g \cdot (C,p_1, \dots , p_n)$, in the sense of Definition \ref{defn:prodsupporder}. Since the points $p_1, \dots , p_n$ are distinct and no three of them lie on a line, at most two of the points $g \cdot p_i$ can lie on the line $\{z=0\}.$ The rest of the points are supported on $z$. Therefore, the term $w_i \min \{ \langle x_i , \lambda_r \rangle \ | \ x_i \in \Xi(g \cdot p_i) \}$ in equation \ref{eq:mu} will equal $w_i \langle z, \lambda_r \rangle = w_i(-1-r)$ for at least $n-2$ values of $i \in \{1, \dots, n\}$.
      
      The curve $C$ is smooth, so $g \cdot C$ cannot be singular at $(1:0:0)$ for any $g \in SL(3)$. By the Jacobian Criterion for smoothness this implies the support of $g \cdot C$ must contain $x^d, x^{d-1}y,$ or $x^{d-1}z$. If $g_0 \in SL(3)$ is a matrix such that $g_0 \cdot C$ is not incident with the point $(1:0:0)$ then $\Xi(g_0 \cdot C)$ will contain the monomial $x^d$ and the maximal support of $g_0 \cdot C$ will therefore be $\{x^d\}$ (Definition \ref{defn:maxsup}). In this case, the maximal support of $g_0 \cdot (C, p_1 \dots , p_n)$ is at most $(\{x^d\}, \{x\}, \{y\}, \{z\}, \dots , \{z\})$, up to a permutation of the points. However, this is not optimal because we can decrease $\Xi_{max}(g \cdot C)$ without decreasing any of the $\Xi_{min}(g \cdot p_i):$
      
      For any $i,j \in \{1, \dots ,n\},$ let $g_1 \in SL(3) $ be a matrix that takes $p_i$ to $(1:0:0)$ and the line $\overline{p_ip_j}$ to $\{z=0\}$. Then, since $(1:0:0) = g_1 \cdot p_i$ lies on $g_1 \cdot C$, the monomial $x^d$ is not in the support of $g_1 \cdot C.$ In fact, we know the maximal support of $g_1 \cdot C$ is $\{x^{d-1}y\}.$ Suppose, by contradiction, $\Xi(g_1 \cdot C)$ does not contain $x^d$ or $x^{d-1}y$. Then it must contain $x^{d-1}z$ by the Jacobian Criterion for smoothness at $(1:0:0).$ But then the tangent cone at $(1:0:0)$ would be the line $\{z=0\},$ however $g_1 \cdot p_i$ and $g_1 \cdot p_j$ lie on this tangent line, contradicting our assumption on the generality of the marked points. Therefore the maximal support of $g_1 \cdot (C, p_1, \dots , p_n)$ is $(\{x^{d-1}y\}, \{x\}, \{y\}, \{z\}, \dots , \{z\}),$ where by abuse of notation $\{x\}$ corresponds to the minimal support of $p_i$ and $\{y\}$ corresponds to the minimal support of $\{p_j\}$. Since this is greater than $\Xi_{max}(g_0 \cdot (C, p_0, \dots , p_n)),$ we may disregard any $g_0 \in SL(3)$ that does not take a point of $C$ to $(1:0:0).$

   We are left to consider the $g_2 \in SL(3)$ such that $x^{d-1}z$ is in the maximal support of $g_2 \cdot C$. Then the maximal support of $g_2 \cdot C$ cannot contain $x^d$ nor $x^{d-1}y$. However $\Xi_{max}(g_2 \cdot C)$ must contain $x^{d-m}y^m$ for some $m \geq 2$; otherwise the line $\{z=0\}$ would be an irreducible factor of $C$ because $z$ would factor out of every monomial, however we know that $C$ is irreducible. If the maximal support of $g_2 \cdot C$ is $\{ x^{d-1}z, x^{d-m}y^m \}$ then the line $\{z=0\}$  intersects $g_2 \cdot C$ at the point $(1:0:0)$ with multiplicity $m$. So $\{z=0\}$ is tangent to $g_2 \cdot C$ at $(1:0:0)$ if $m=2$ and inflectional to $C$ at $(1:0:0)$ if $2< m \leq d$. 
   If $m>2$ then the minimal support of all the points is $\{z\}$ because none of the points can lie on the flex line $\{z=0\}$. If $m=2$ then the minimal support of the points $p_1, \dots , p_n$ is at most $\{x\}, \{z\}, \dots , \{z\}$ up to permutation, because no two points lie on the tangent line $\{z=0\}$. Therefore, the maximal values of $\Xi_{max}(g \cdot (C, p_1 , \dots , p_n))$ are
   \begin{itemize}
       \item $(\{x^{d-m}y^m, x^{d-1}z\}, \{z\}, \dots , \{z\})$, $2 < m \leq d$, 
       \item $(\{x^{d-2}y^2, x^{d-1}z\}, \{x\}, \{z\}, \dots , \{z\})$,
       \item and $(\{x^{d-1}y\}, \{x\}, \{y\}, \{z\}, \dots , \{z\})$
   \end{itemize}
   up to a permutation of the points.

\begin{enumerate}[label= (Case \arabic*):]
    \item Suppose the maximal support of $g_2 \cdot (C, p_1, \dots , p_n)$ is $(\{x^{d-m}y^m, x^{d-1}z\}, \{z\}, \dots , \{z\})$, for $2 < m \leq d$. Then using equation \ref{eq:mu},
    $$\mu^L (g_2 \cdot (C, p_1, \dots , p_n), \lambda_r) = (-1-r) (\sum_1^n w_i) - \gamma \max \begin{cases} (d-m) + mr \\
        d-2-r
    \end{cases}$$
    $$= - \sum_1^n w_i - \gamma d + \begin{cases}
       \gamma m + r(- \sum_1^n w_i - \gamma m)  \ \ \ \text{if} \ r \geq \frac{m-2}{m+1} \\
         2 \gamma  + r(- \sum_1^n w_i + \gamma)\ \ \ \text{if} \ r \leq \frac{m-2}{m+1}
    \end{cases}.$$
Because the coefficient of $r$ is negative when $r \geq \frac{m-2}{m+1},$ the above numerical function is maximized for $r \in [\frac{-1}{2}, 1]$ when $r$ equals $\frac{m-2}{m+1}$ or $\frac{-1}{2},$ giving maximal values \\
 $F_1 := \mu^L (g_2 \cdot (C, p_1, \dots , p_n), \lambda_{\frac{m-2}{m+1}} ) = \frac{-2m+1}{m+1} \sum_1^n w_i - \gamma d + \frac{ 3 \gamma m}{m+1}, $ and \\
 $F_2 := \mu^L (g_2 \cdot (C, p_1, \dots , p_n), \lambda_{\frac{-1}{2}} ) = \frac{-1}{2}\sum_1^n w_i - \gamma d + \frac{3}{2} \gamma .$

\item Suppose the maximal support of $g_2' \cdot (C, p_1, \dots , p_n)$ is $(\{x^{d-2}y^2, x^{d-1}z\}, \{z\}, \dots , \{x\}, \dots , \{z\})$, where $\{x\} = \Xi_{min}(p_i)$. Then 
$$\mu^L (g_2' \cdot (C, p_1, \dots , p_n), \lambda_r) = w_i + ((\sum_{j=1}^n w_j) - w_i)(-1-r) - \gamma \max \begin{cases}
    d-2 +2r \\
    d-2-r
\end{cases}$$


$$ = 2w_i - \sum_{j =1}^n w_j - \gamma d + 2 \gamma + \begin{cases}
    r(w_i - \sum_{j=1}^n w_j - 2 \gamma) \ \ \ \text{if} \ r \geq 0 \\
     r(w_i - \sum_{j=1}^n w_j + \gamma) \ \ \ \text{if} \ r \leq 0 
\end{cases}.$$
This is maximized when $r$ equals $\frac{-1}{2}$ or $0$, giving maximal values \\
$F_3 := \mu^L (g_2' \cdot (C, p_1, \dots , p_n), \lambda_{\frac{-1}{2}})= \frac{3}{2}w_i - \frac{1}{2}\sum_{j=1}^n w_j + \frac{3}{2} \gamma -  \gamma d$, and \\
$F_4 := \mu^L (g_2' \cdot (C, p_1, \dots , p_n), \lambda_0) = 2w_i - \sum_{j=1}^n w_j - \gamma d + 2 \gamma.$

\item  Suppose the maximal support of $g_1 \cdot (C, p_1, \dots , p_n)$ is $(\{x^{d-1}y\},  \{z\}, \dots , \{x\}, \dots , \{y\}, \dots ,\{z\})$, where $\{x\} = \Xi_{min}(p_i)$ and $\{y\} = \Xi_{min}(p_j)$. Then
$$\mu^L (g_1 \cdot (C, p_1, \dots , p_n), \lambda_r) = w_i + rw_j   + (-1-r)((\sum_{k=1}^n w_k)- w_i - w_j) - \gamma (d-1 + r)$$
$$ = 2w_i + w_j - \sum_1^n w_k  - \gamma d + \gamma       + r(w_i + 2w_j - \sum_1^n w_k - \gamma   ).$$

This function is linear in $r$ so it is maximized at one of the endpoints, $r=1$ or $r= \frac{-1}{2},$ giving maximal values \\
$F_5:=\mu^L ( g_1 \cdot (C, p_1, \dots , p_n), \lambda_1) = 3w_i + 3w_j - 2 \sum_1^n w_k - \gamma d,$ and \\
$F_6 :=\mu^L ( g_1 \cdot (C, p_1, \dots , p_n), \lambda_{\frac{-1}{2}}) = \frac{3}{2}w_i - \frac{1}{2}\sum_1^n w_k + \frac{3}{2} \gamma - \gamma d.$
    
\end{enumerate}

The functions $F_1 (\gamma, w_1, \dots , w_n) , \dots, F_6 (\gamma, w_1, \dots , w_n)$ are the maximal values of $\mu^L(g \cdot (C, p_1, \dots , p_n), \lambda_r)$ for all $g \in SL(3)$ and $r \in [\frac{-1}{2},1]$. The marked curve $(C, p_1, \dots , p_n)$ is therefore stable where these functions are all negative. By our assumptions that the weights $\gamma, w_1, \dots , w_n$ are all positive, $2 < m \leq d,$ and $d \geq 3$, the functions $F_1$ and $F_2$ are always negative. So the locus of line bundles in $NS^{SL(3)}(\mc{C}_{n,d})_{\Q}$ for which $(C, p_1, \dots , p_n)$ is stable is given by the three families of inequalities, for each $i,j \in \{1, \dots n\}, i \neq j$, $F_3 = F_6 < 0; F_4 < 0;$ and $F_5 < 0  $. Clearing denominators gives $3w_i - \sum_1^k w_k + 3 \gamma -2 \gamma d < 0;$ $2w_i - \sum_{k=1}^n w_k - \gamma d + 2 \gamma < 0;$ and $ 3w_i + 3w_j - 2 \sum_1^n w_k - \gamma d < 0$, for each distinct pair $i,j \in \{1, \dots , n\}$.
 \end{proof}

\subsection{Preliminary results on the geometry of plane cubics} \label{sec:cubicgeom}

\begin{lemma}\label{lemma:classofcubics}
  Let $C$ be a plane cubic. If $C$ is nonsingular, it is projectively equivalent to $\{z^2x + y^3 + ax^2y + bx^3=0 \}$ for some $a,b \in \C$ with $4a^3 + 27b^2 \neq 0$. If $C$ is singular then it is projectively equivalent to one of the following 
\begin{align*}
\{xz^2 + y^3 + xy^2 =0 \} && \{xz^2 + y^3= 0\}  && \{xyz+y^3 = 0\} && \{x^2z - xy^2=0\}
\end{align*}
  \begin{align*}
  \{ xyz = 0\} && \{ x^2y + xy^2 =0\} && \{x^3 + x^2y=0  \} && \{x^3 = 0\}.    
  \end{align*}
  These curves are depicted with their singularities labeled in Figure \ref{fig:cubics}.
\end{lemma}

\begin{proof}
Plane cubics are classified in detail in Chapter 10.3 of \cite{Dol03} .
\end{proof}

As a result of Lemma \ref{lemma:stab}, the GIT inner-walls correspond to marked plane curves $\mathbf{x}$ with positive dimensional $SL(3)$ stabilizer $SL(3)_\mathbf{x}$. Suppose $\mathbf{x} = (C, p_1, \dots , p_n)$ has a positive dimensional stabilizer. Then the curve $C$ in particular must have positive dimensional stabilizer. In the case of curves of degree $d=3$ there are only five reduced curves with positive dimensional stabilizer, whereas non-reduced marked cubics are unstable for all $L$ and therefore do not give rise to GIT walls.

\begin{lemma}\label{lemma:CubicPositStbz}
 Let $C$ be a reduced plane cubic with positive dimensional stabilizer. Then $C$ is projectively equivalent to one of the following:  
 \begin{align*}
   \{xz^2 + y^3 =0\} &&  \{xyz+y^3 = 0\} && \{x^2z - xy^2 =0\} && \{xyz =0\} && \{x^2y+xy^2 = 0\}.
 \end{align*}
\end{lemma}

\begin{proof}
    \cite[Section 2]{Ple99}.
\end{proof}

\begin{figure}\label{figure: curvedepict}
    \centering
    \includegraphics[scale=0.4]{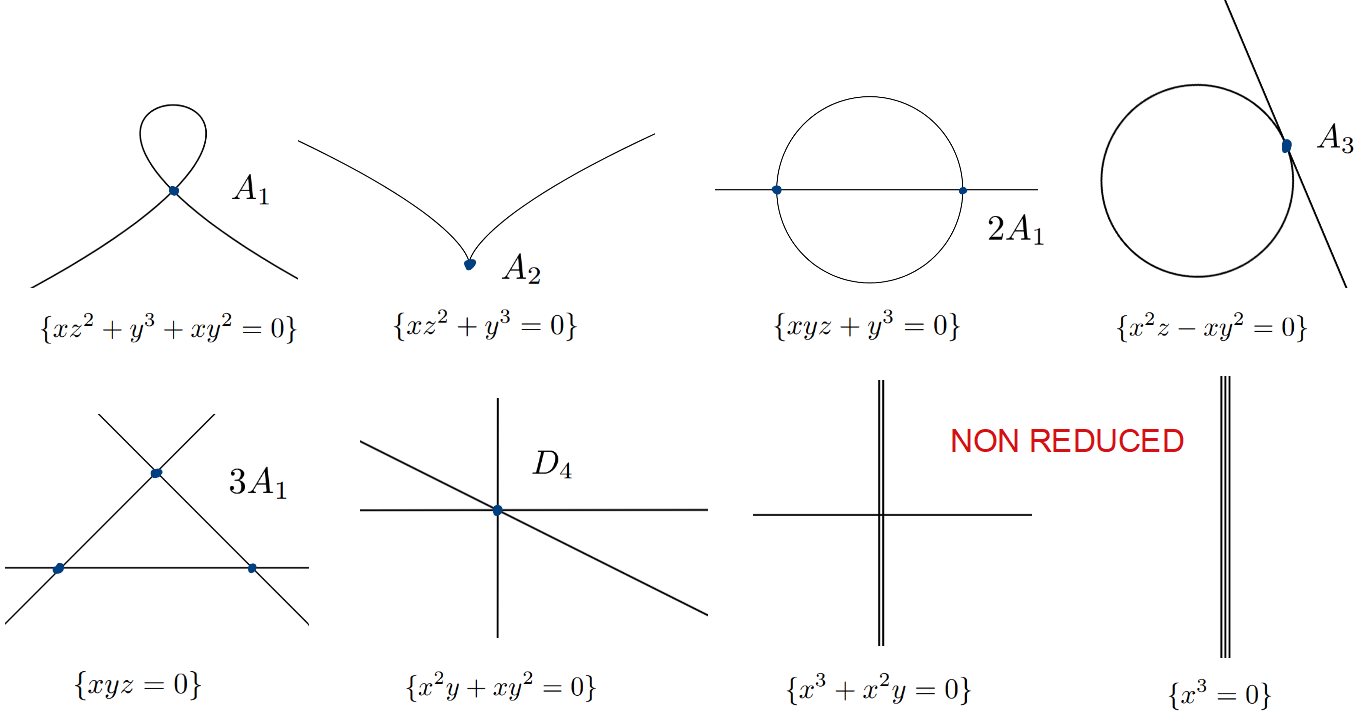}
    \caption{Singular cubics are determined up to projective equivalence by their singularities}
    \label{fig:cubics}
\end{figure}

We now add the marked points into consideration, finding the marked curves $(C,p_1, \dots p_n)$ with positive-dimensional $SL(3)$ stabilizer when $C$ is one of the five curves in Lemma \ref{lemma:CubicPositStbz}.

\begin{lemma}\label{lemma:CuspStbz}
    If the marked curve $(\{xz^2 + y^3 =0\}, p_1, \dots , p_n)$ has positive dimensional $SL(3)$ stabilizer then for each $i$, the marked point $p_i$ is either at the cuspidal singularity $(1:0:0)$ or at the smooth flex point $(0:1:0)$.
\end{lemma}

\begin{proof}
    The stabilizer of $\{xz^2 + y^3=0\}$ is given in \cite[Section 2]{Ple99} as a $\C^*$ subgroup of $SL(3)$ equal to the image of the one-parameter subgroup $\diag\{5,-1,-4\}$. It is straightforward to check that the only points fixed by this subgroup are the points $(1:0:0)$, $(0:1:0)$, and $(0:0:1)$. Of these points, $(1:0:0)$ and $(0:1:0)$ lie on the cuspidal curve $\{xz^2 + y^3=0\}$.
\end{proof}

\begin{lemma}\label{lemma:ConicTransStbz}
    If the marked curve $(\{xyz +y^3 =0\}, p_1, \dots , p_n)$ has positive dimensional $SL(3)$ stabilizer then for each $i$, the marked point $p_i$ is at one of the two nodal singularities: $(1:0:0)$ or $(0:0:1)$.
\end{lemma}

\begin{proof}
    The stabilizer of $\{xyz +y^3 =0\}$ is given in  \cite[Section 2]{Ple99} as a $\C^*$ subgroup of $SL(3)$ equal to the image of the one-parameter subgroup $\diag\{1,0,-1\}$. As in the above lemma, the only points fixed by this subgroup are the points $(1:0:0)$, $(0:1:0)$, and $(0:0:1)$. Of these points, $(1:0:0)$ and $(0:0:1)$ lie on the curve $\{xyz +y^3 =0\}$.
\end{proof}

\begin{lemma}\label{lemma:ConicTangStbz}
    If the marked curve $(\{x^2z - xy^2 =0\}, p_1, \dots , p_n)$ has positive dimensional $SL(3)$ stabilizer then there exists a line $l=\{z=-k^2x +2ky\}$ for some $k \in \C$, tangent to the conic, such that for each $i$, the marked point $p_i$ is either at the $A_3$ singularity $(0:0:1)$ or
    on the line $l$.
\end{lemma}

\begin{proof}
    We compute the stabilizer of $C:=\{x^2z - xy^2 =0\}$ directly. Let $M \in SL(3)$ be a matrix fixing $\{x^2z - xy^2 =0\}$. Since the irreducible components of $\{x^2z - xy^2 =0\}$ have degrees $1$ and $2$, the linear automorphism $M$ must fix the conic $\{xz -y^2 =0\}$ and fix the line $\{x=0\}$, and therefore must fix their intersection point $(0:0:1)$. The matrices fixing $\{x=0\}$ and $(0:0:1)$ are those which are lower triangular. Since $M$ is in $SL(3)$ it has determinant equal to $1$ and must therefore be of the form
    $$M =\begin{bmatrix}
        a & 0 & 0 \\
        b & c & 0 \\
        d & e & \frac{1}{ac}
    \end{bmatrix}, \ a,c \neq 0.$$

The fact that $M$ fixes the conic $\{xz=y^2\}$ implies that $d= \frac{b^2}{a}, e= \frac{2bc}{a},$ and $c$ is a cube-root of unity. Indeed, the matrices
$$\bigg\{ \begin{bmatrix}
        a & 0 & 0 \\
        b & c & 0 \\
        \frac{b^2}{a} & \frac{2bc}{a} & \frac{1}{ac}
    \end{bmatrix} \ |  \ a,c \neq 0, \ c^3=1 \bigg\}$$
all fix the curve $C.$

If $M$ fixes a point $(x_0:y_0:z_0) \neq (0:0:1)$ then the linear equations $M \cdot \langle x_0, y_0, z_0 \rangle = \langle \lambda x_0, \lambda y_0, \lambda z_0 \rangle, \ \lambda \in \C^*$ and the cubic equation $x_0^2z_0-x_0y_0^2 =0$ can be solved to give $(x_0:y_0:z_0) = (0:1:\frac{2b}{a-c})$ or $(1:\frac{b}{a-c}:(\frac{b}{a-c})^2)$. If a positive-dimensional family of the stabilizer of $C$ fixes either of these points then $b$ must be a constant multiple of $a-c$. Then, for any $k \in \C$, taking $b := k(a-c)$, $c$ a cube-root of unity, and $a:=t$ a free parameter gives a one-dimensional subgroup of $SL(3)$ fixing the curve $C$ and the points $(0:0:1), (0:1:2k),$ and $(1:k:k^2),$
$$\bigg\{ \begin{bmatrix}
        t & 0 & 0 \\
        k(t-c) & c & 0 \\
        \frac{k^2(t-c)^2}{t} & \frac{2k(t-c)c}{t} & \frac{1}{tc}
    \end{bmatrix} \ | \  t \in \C^* \bigg\}.$$
Furthermore, we have shown that any one-dimensional family of matrices fixing $C$ and some point other than $(0:0:1)$ must be of this form. The fixed points $(0:1:2k)$ and $(1:k:k^2)$ lie on the line $l$ which is tangent to $C$ at  $(1:k:k^2)$.
\end{proof}

\begin{lemma}\label{lemma:TriangleStbz}
    If the marked curve $(\{xyz =0\}, p_1, \dots , p_n)$ has positive dimensional $SL(3)$ stabilizer then there is a permutation of coordinates $\{x_i, x_j, x_k\} := \{x,y,z\}$ such that for each $i$, the marked point $p_i$ is either at the node $\{x_i=x_j=0\}$ or on the opposite line $\{x_k=0\}$. 
\end{lemma}

\begin{proof}
    The stabilizer of $\{xyz=0\}$ is our diagonal torus $T$. A generic one-parameter subgroup $\diag\{a,b,-a-b\}$ fixes only the torus-invariant points $(1:0:0), (0:1:0),$ and $(0:0:1)$. However, if $a=b$ then $\diag\{a,a,-2a\}$ fixes all the points on the line $z=0$ in addition to the point $(0:0:1)$. Similar behavior occurs in the other two exceptional cases, when $a=-a-b$ and when $b=-a-b$.
\end{proof}

\begin{lemma}\label{lemma:D4Stbz}
    If the marked curve $(\{x^2y +xy^2 =0\}, p_1, \dots , p_n)$ has positive dimensional $SL(3)$ stabilizer then there exists a line $l$ such that for each $i$, the marked point $p_i$ is either at the singular point $(0:0:1)$ or on $l$.
\end{lemma}

\begin{proof}
    The curve $C:=\{x^2y+xy^2=0\}$ contains three linear components, $\{x=0\}, \{y=0\}, $ and $ \{x+y=0\}$. Any element of its stabilizer must therefore permute these lines and fix their point of intersection $(0:0:1)$. If $g \in SL(3)$ fixes at least three points on one of these lines then it must fix the entire line because a linear automorphism of $\Pj^1$ is determined by the image of three points. Indeed, there are one-dimensional families of matrices that fix $C$ and all of the points on one of its linear components: Without loss of generality we may assume the fixed linear component is $\{x=0\}$. Then the families 
    $$\bigg\{ \begin{bmatrix}
        1 & 0 & 0 \\
        0 & 1 & 0 \\
        t & 0 & 1
    \end{bmatrix} \ | \ t \in \C^* \bigg\} \ \ \text{and} \ \ \bigg\{ \begin{bmatrix}
        \frac{1}{c^2} & 0 & 0 \\
        \frac{-1}{c^2} & c & 0 \\
        t & 0 & c
    \end{bmatrix} \ | \ t \in \C^*, \ c^3=-1\bigg \}$$ fix $C$ and the points on $\{x=0\}$. \\

    If, however, $g$ fixes points on two distinct linear components, other than the point $(0:0:1)$, then $g$ must permute the linear components trivially. The matrices that induce automorphisms of the lines $\{x=0\}$, $\{y=0\}$, and $\{z=0\}$ are of the form
    $$M = \begin{bmatrix}
        a & 0 & 0 \\
        0 & a & 0 \\
        b & c & \frac{1}{a^2}
    \end{bmatrix}.$$
Let $p \in C$ be a point fixed by $M$. Assume $p \neq (0:0:1)$. Then $M\cdot p = p$ implies that if $p \in \{x=0\}$ then $p= (0:1: \frac{a^2c}{a^3-1})$ , if $p \in \{y =0\}$ then $p = (1:0:\frac{a^2b}{a^3-1}),$ and if $p \in \{x+y = 0\}$ then $p = (1:1: \frac{a^2b + a^2c}{a^3-1})$. If there is a one-dimensional family of $SL(3)$  fixing two of these points then $a^2b$ must be a constant multiple $k$ of $a^3-1$ and $a^2c$ must be a constant multiple $k'$ of $a^3-1$ and, in fact, all three of these points will be fixed by such a family of matrices. The three fixed points $(0:1:k'), (1:0:k),$ and $(1:1:k+k')$ all lie on the line $z=kx+k'y$. We can take any $k,k' \in \C$ to construct a one-dimensional family in $SL(3)$ fixing the curve $C$, the point $(0:0:1),$ and the points $ (1:0:k), (0:1:k'),$ and $(1:1:k+k')$ on the line $z=kx+k'y$:
$$\bigg\{ \begin{bmatrix}
    t & 0 & 0 \\
    0 & t & 0 \\
    kt - \frac{k}{t^2} & k't - \frac{k'}{t^2} & \frac{1}{t^2}
\end{bmatrix} \ | \ t \in \C^* \bigg\}.$$

Therefore if a set of points is fixed by a one-dimensional family of $Stab_{SL(3)}(C)$ they must all either be at the origin $(0:0:1)$ or on some line $k''z= kx+k'y$. Note, the case where $k''=0$ corresponds to the first case of this proof in which a subgroup fixes a single linear component of $C$. 
    
\end{proof}


\section{ The SL(3)-ample cone} \label{subsec:outerwalls}

In order to prove Theorem \ref{thm:ample_cone}, we first establish a Lemma that $\Lambda(\mc{C}_{n,d})$ corresponds to the locus of $NS^{SL(3)}(\mc{C}_{n,d})_\Q$ at which a sufficiently general plane curve is semi-stable.

\begin{lemma}\label{lemma:semistable_generic}
    Let $(C, p_1 , \dots , p_n)$ be a marked plane curve of degree $d \geq 3$ that is generic in the sense of Theorem \ref{thm:generic}. Then $(C, p_1 , \dots , p_n)$ is semi-stable on the entirety of $\Lambda(\mc{C}_{n,d})$. Furthermore, if any curve $(D, q_1, \dots, q_n)$ is $L$ stable for some $L \in \Lambda(\mc{C}_{n,d})$ then $(C, p_1 , \dots , p_n)$ is as well.
\end{lemma}

\begin{proof}
Let $L = (\gamma, w_1, \dots , w_n)$ be an $SL(3)-$linearized line bundles in $\Lambda(\mc{C}_{n,d})$. We show that $\mu^L(h \cdot (C, p_1, \dots , p_n), \lambda_r)$ is non-positive for all $h \in SL(3)$ and normalized one-parameter subgroups $\lambda_r$. In Theorem \ref{thm:generic}, we saw that as $h$ varies in $SL(3)$, the maximal values of $\Xi(h \cdot (p_1, \dots , p_n))$ are 
\begin{enumerate}[label=\Alph*]
    \item := $(\{x^{d-m}y^m, x^{d-1}z\}, \{z\}, \dots , \{z\})$ for $2 < m \leq d$, 
    \item := $(\{x^{d-2}y^2, x^{d-1}z\}, \{x\}, \{z\}, \dots , \{z\})$, and
    \item := $(\{x^{d-1}y\}, \{x\}, \{y\}, \{z\}, \dots , \{z\})$,
\end{enumerate}
up to a permutation of the order of the marked points. We show that in the above three cases, $\mu^L(h \cdot (C, p_1, \dots , p_n), \lambda_r)$ is negative for all $r \in [\frac{-1}{2}, 1]$. First, as calculated in Theorem \ref{thm:generic}, the assumptions that the weights $\gamma, w_1, \dots , w_n$ are all positive, $2 < m \leq d,$ and $d \geq 3$, imply that if the maximal support of $h \cdot (C, p_1, \dots , p_n)$ is A, then $\mu^L(h \cdot (C, p_1, \dots , p_n), \lambda_r)$ is negative.

Now assume that the maximal support of $h \cdot (C, p_1, \dots , p_n)$ is B or C. Since $L$ is in the $SL(3)-$ample cone, there exists some marked curve $(D, q_1, \dots , q_n)$ which is $L$ semi-stable. Then we can find a linear automorphism $h_1 \in SL(3)$ taking $q_1$ to $(1:0:0)$ such that $[z=0]$ is tangent to $h_1 \cdot D$ at $h_1 \cdot q_1$. This geometric property implies the maximal support of $h_1 \cdot (D, q_1, \dots , q_n)$ is at least $B$. Likewise, we can find a linear automorphism $h_2 \in SL(3)$ taking $q_1$ to $(1:0:0)$ and $q_2$ to a point on the line $\{z=0\}$ so that the maximal support of  $h_2 \cdot (D, q_1, \dots , q_n)$ is at least C. Therefore the maximal support of $h \cdot (C, p_1, \dots , p_n)$ is less than or equal to $max_{g \in SL(3)} (\Xi_{max}(g \cdot (D, q_1, \dots , q_n)))$. Since  $(D, q_1, \dots , q_n)$ is $L$ semi-stable, $\max_{g \in SL(3)} \mu^L (g \cdot (D, q_1, \dots , q_n), \lambda_r) \leq 0$. It then follows that $\mu^L(h \cdot (C, p_1, \dots , p_n), \lambda_r)$ is non-positive.
$$\max_{r\in [\frac{-1}{2},1]} \mu^L(h \cdot (C, p_1 , \dots , p_n), \lambda_r) \leq \max_{g \in SL(3), r \in [\frac{-1}{2},1]} \mu^L (g \cdot (D, q_1, \dots , q_n), \lambda_r) \leq 0.$$
We conclude that the maximum value of $\mu^L(h \cdot (C, p_1, \dots , p_n), \lambda_r)$ is non-positive for all $h \in SL(3)$. If $(D, q_1, \dots q_n)$ is $L$ stable then the above inequality on the right becomes strict, making the maximum value of $\mu^L(h \cdot (C, p_1, \dots , p_n), \lambda_r)$ negative.
\end{proof}

At this point we have done all the leg work for Theorem \ref{thm:ample_cone}.

\begin{thm}
\label{thm:ample_cone}
    The cone $\Lambda(\mc{C}_{n,d})$ is

\begin{align*}
 \bigg\{
(\gamma, w_1, \ldots, w_n)
\; \bigg| \;
   w_i \leq \frac{W+\gamma(2d-3)}{3} , \ w_i \leq \frac{W+ \gamma(d-2)}{2}, \ w_i+w_j \leq \frac{2W+ \gamma d}{3}, \ 
   0 < w_i, \ 0 < \gamma
 \bigg\},
\end{align*}
where $W = \sum_1^n w_i$ denotes the total weight of the points.
\end{thm}

\begin{proof}[Proof of Theorem \ref{thm:ample_cone}]

Let $(C, p_1, \dots , p_n)$ be sufficiently general in the sense that the $p_i$ are distinct; no point lies on a flex line of $C$; no three points lie on a line; and no two points lie on a tangent line of $C$. By Lemma \ref{lemma:semistable_generic},  $\Lambda(\mc{C}_{n,d})$
 is equal to the locus of line bundles for which $(C, p_1, \dots , p_n)$ is semi-stable.  In Theorem \ref{thm:generic} we computed this to be given by three families inequalities: $3w_i - \sum_1^k w_k + 3 \gamma -2 \gamma d \leq 0;$ $2w_i - \sum_{k=1}^n w_k - \gamma d + 2 \gamma \leq  0;$ and $3w_i + 3w_j - 2 \sum_1^n w_k - \gamma d \leq 0$, for each pair $i,j \in \{1, \dots , n\}$. We solve for $w_i$ in the first two inequalities and for $w_i + w_j$ in the final inequality to get stability conditions for the weights of individual points.
   
\end{proof}

\begin{cor}
    Any general curve $(C, p_1 , \dots , p_n)$ as in Theorem \ref{thm:generic} is stable on the interior of $\Lambda(\mc{C}_{n,d})$.
\end{cor}
\begin{proof}
 This is because the maximal values of $\mu^L(g \cdot (C, p_1, \dots , p_n), \lambda_r)$ as computed in Theorem \ref{thm:generic} are negative on the interior of the $\Lambda(\mc{C}_{n,d})$.
\end{proof}



\section{Finding the inner walls} \label{sec:innerwalls}

We now compute the \textit{inner walls} of the cone $\Lambda(\mc{C}_{n,d})$. The following lemma associates to each inner wall a marked plane curve with positive-dimensional $SL(3)$ stabilizer.

\begin{lemma}\label{lemma:stab}
    If $(D, q_1, \dots , q_n)$ is strictly semi-stable with respect to $L$ then there exists a marked curve $(C, p_1, \dots , p_n)$ with positive dimensional stabilizer in the orbit closure $\overline{SL(3) \cdot (D, q_1, \dots , q_n)}$ which is strictly semi-stable with respect to $L$. 
\end{lemma}

\begin{proof}
Suppose $(D, q_1, \dots , q_n)$ is strictly semi-stable with respect to $L$. Since $(D, q_1, \dots , q_n)$ is not stable, it either has an open orbit in $\mc{C}_{n,d}^{ss}(L)$ or it has positive dimensional stabilizer. There exists a unique closed orbit $O$ in $\overline{SL(3) \cdot (D, q_1, \dots , q_n)}$. Let $(C, p_1, \dots , p_n)$ be a representative of $O$. Then since $(C, p_1, \dots , p_n)$ has a closed orbit and is strictly semi-stable at $L$ it must have a positive dimensional stabilizer.
\end{proof}

Due to Lemma \ref{lemma:stab}, to find all the GIT walls it suffices to find the line bundles in $\Lambda(\mc{C}_{n,d})$ at which a marked curve with positive-dimensional stabilizer is strictly semi-stable. Such marked curves were classified in Section \ref{sec:cubicgeom}.

We begin with the case of three non-concurrent lines, that is, a curve projectively equivalent to $\{xyz=0\}$. Lemma \ref{lemma:TriangleStbz} shows that if $(C, p_1, \dots , p_n)$ has positive-dimensional stabilizer and $C$ is three non-concurrent lines then the marked points are either at a node $\eta$ of $C$ or on the opposite linear component $\Lambda \subset C$.

\begin{lemma}[Case of $C(3A_1)$]
\label{lemma:Wall_3distinctlines}
    Let $\mathbf{x}:=(C, p_1, \dots p_n)$ have positive dimensional stabilizer, where the curve $C$ is three lines in general position. Let $\eta$ and $\Lambda$ be as above. Define $I$ to be the subset of $\{1, \dots , n\}$ such that the points $p_i$ coincide at $\eta \in C$ for $i \in I$ and the remaining points lie on the opposite line $\Lambda \subset C$. Then $\mathbf{x}$ is $L$ strictly semi-stable if and only if $\sum_{i \in I}w_i = \frac{1}{2} \sum_{j \notin I} w_j$, the total weight at each of the other nodes is at most $\frac{1}{3} \sum_1^n w_i$, and the weight at any smooth point of $\Lambda$  is at most $\frac{1}{3}\sum_1^n w_i + \gamma $.

\end{lemma}

\begin{proof}
We solve $\max \{\mu^L(g \cdot \mathbf{x}, \lambda_r) \ | \ g \in SL(3), r \in [\frac{-1}{2},1] \} = 0$ for $L$ to find the line bundles at which the marked curve is strictly semi-stable. As discussed in Section 2.3 our method is to identify the $g \in SL(3)$ which maximize $\Xi_{max}(g \cdot \mathbf{x})$ because such $g$ will also maximize $\mu^L (g \cdot \mathbf{x}, \lambda_r)$.

There are six cases. First we consider the $g \in SL(3)$ which take the node $\eta$ to $(1:0:0)$. Under this assumption there are two ways to maximize $\Xi_{max}(g \cdot \mathbf{x})$.  A nodal singularity of a plane cubic at the origin $(1:0:0)$ of the affine chart $U_{x=1}$ is characterized by the Jacobian Criterion and the fact that the tangent cone has two distinct rays. That is, if $f$ is a polynomial with $V(f) = g \cdot C$ then $f(1,y,z) = ay^2 + byz + cz^2 + h_3(y,z)$, where $ay^2 + byz + cz^2$ is not a perfect square and $h_3(y,z)$ is a homogeneous cubic. If $g_1$ takes the line through $\eta$ and one of the other nodes $\eta'$ to $\{z=0\}$ then  $\Xi_{max}(g_1 \cdot C)$ is $\{xyz\}$. This is because one of the tangent lines at $(1:0:0)$ will be $\{z=0\}$, so $g_1 \cdot C$ is not supported on $x^3, x^2y, $ or $x^2z$ by the Jacobian criterion and it must have $xyz$ but not $xy^2$ in its support to have $\{z=0\}$ and one other line in its tangent cone at the origin. In this case $\Xi_{min}(g_1 \cdot p_i) = \{x\}$ for $p_i = \eta$, $\Xi_{min}(g_1 \cdot p_j) = \{y\}$ for $p_j = \eta'$, and $\Xi_{min}(g_1 \cdot p_k) = \{z\}$ for the remaining points. The other maxima of $\Xi_{max}(g \cdot \mathbf{x})$ when $g \cdot \eta = (1:0:0)$ correspond to $g_2 \in SL(3)$ which take the line through $\eta$ and $p_j$ to $\{z=0\}$ where $p_j$ is a smooth point on $\Lambda$ at which the most weight is concentrated. Then $\Xi_{max}(g_2 \cdot C) = \{xy^2\}$ because the line $\{z=0\}$ is not in the tangent cone at $(1:0:0)$.

Next, assume that $g \in SL(3)$ does not take $\eta$ to $(1:0:0)$. Under this assumption there are four ways to maximize $\Xi_{max}(g \cdot \mathbf{x})$ as follows: In all cases, $\Xi_{max}(g \cdot \mathbf{x})$ is maximized when $g$ takes a node $\eta'$ or one of the points $p_j \in \Lambda$ to $(1:0:0)$. If $g_3$ takes $\eta'$ to $(1:0:0)$ and the line $\overline{\eta \eta'}$ to $\{z=0\}$ then $\Xi_{max}(g_3 \cdot C) = \{xyz\}$, similar to $g_1 \cdot C$ above. If $g_4$ takes $\eta'$ to $(1:0:0)$ and $\Lambda$ to $\{z=0\}$ then $\Xi_{max}(g_4 \cdot C)$ is also $\{xyz\}$ although the minimal supports of the points changes. Let $p_j$ be a smooth point on $\Lambda$ at which the most weight is concentrated. If $g_5$ takes $p_j$ to $(1:0:0)$ and the line $\overline{p_j \eta }$ to $\{z=0\}$ then, by the Jacobian Criterion and the fact that $\{z=0\}$ is not tangent to $C$ at $(1:0:0)$, $\Xi_{max}(g_5 \cdot C) $ is $\{x^2y\}$. Finally, if $g_6$ takes $p_j$ to $(1:0:0)$ and $\Lambda$ to $\{z=0\}$ then $\Xi_{max}(g_6 \cdot C)$ is $\{x^2z\}$.

We abbreviate $\sum_{p_i = \eta} w_i$ as $w_{\eta}$, $\sum_{p_i = \eta'} w_i$ as $w_{\eta'}$, $\sum_{p_i \in \Lambda} w_i$ as $w_{\Lambda}$, and $\sum_{p_i = p_j} w_i$, where $p_j$ is the smooth point on $\Lambda$ at which the most weight is concentrated, as $w_{J}$ . The maximal values of $\mu^L( g \cdot \mathbf{x}, \lambda_r)$, corresponding to the maximal values of $\Xi_{max} (g \cdot \mathbf{x})$ (Lemma \ref{lemma:maxmuxi}) found above, are:

\begin{itemize}
    \item $\mu^L(g_1 \cdot \mathbf{x}, \lambda_r) = w_\eta + r w_{\eta'} + (-1-r)(w_\Lambda - w_{\eta'})$, which is maximized in $r$ by  
    \begin{align*} F_1 := \mu^L (g_1 \cdot \mathbf{x}, \lambda_{\frac{-1}{2}}) = w_\eta - \frac{1}{2} w_\Lambda &&
    \text{and} && F_2:= \mu^L (g_1 \cdot \mathbf{x}, \lambda_1) = w_\eta+ 3 w_\eta' - w_\Lambda.\end{align*}

    \item $\mu^L(g_2 \cdot \mathbf{x}, \lambda_r) = w_\eta + r w_J + (-1-r)(w_\Lambda - w_J) - \gamma(1+2r), $ which is maximized in $r$ by
\begin{align*}
    F_3:= \mu^L(g_2 \cdot \mathbf{x}, \lambda_{\frac{-1}{2}}) = w_\eta - \frac{1}{2}w_\Lambda
    && \text{and} && F_4:=\mu^L(g_2 \cdot \mathbf{x}, \lambda_1) = w_\eta - 2w_\Lambda + 3w_J - 3 \gamma.
\end{align*}
    
    \item $\mu^L(g_3 \cdot \mathbf{x}, \lambda_r) = w_{\eta'}  + rw_\eta + (-1-r)(w_\Lambda - w_{\eta'}),$ which is maximized in $r$ by

    \begin{align*}
      F_5:=  \mu^L(g_3 \cdot \mathbf{x}, \lambda_{\frac{-1}{2}}) = \frac{3}{2}w_{\eta'} - \frac{1}{2}w_\eta - \frac{1}{2} w_\Lambda
      && \text{and} && F_6:= \mu^L(g_3 \cdot \mathbf{x}, \lambda_1) = 3 w_{\eta'} + w_\eta - 2 w_\Lambda .
    \end{align*}
    
    \item $\mu^L(g_4 \cdot \mathbf{x}, \lambda_r) =  w_{\eta'} + r(w_\Lambda - w_{\eta'}) + (-1-r)w_\eta,$
    which is maximized in $r$ by
    \begin{align*}
        F_7 := \mu^L(g_4 \cdot \mathbf{x}, \lambda_{\frac{-1}{2}}) = \frac{3}{2}w_{\eta'} - \frac{1}{2}w_\eta - \frac{1}{2}w_\Lambda
        && \text{and} && 
        F_8:= \mu^L(g_4 \cdot \mathbf{x}, \lambda_1) = w_\Lambda - 2 w_\eta.
    \end{align*}

    \item $\mu^L(g_5 \cdot \mathbf{x}, \lambda_r) = w_J + r w_\eta + (-1-r)(w_\Lambda - w_J) - \gamma (2+r)$
     which is maximized in $r$ by
     \begin{align*}
         F_9:= \mu^L(g_5 \cdot \mathbf{x}, \lambda_{\frac{-1}{2}}) = \frac{3}{2}w_J - \frac{1}{2}w_\Lambda - \frac{3}{2} \gamma - \frac{1}{2}w_\eta
         && \text{and} &&
         F_{10} := \mu^L(g_5 \cdot \mathbf{x}, \lambda_1) = 3 w_J - 2w_\Lambda -3\gamma + w_\eta.
     \end{align*}
    
    \item $\mu^L(g_6 \cdot \mathbf{x}, \lambda_r) = w_J + r(w_\Lambda - w_J) + (-1-r)w_\eta - \gamma(1-r) $
    which is maximized in $r$ by
    \begin{align*}
        F_{11} : = \mu^L(g_6 \cdot \mathbf{x}, \lambda_{\frac{-1}{2}}) = \frac{3}{2}w_J - \frac{1}{2}w_\eta - \frac{3}{2}\gamma - \frac{1}{2}w_{\Lambda}
        && \text{and} &&
           F_{12} : = \mu^L(g_6 \cdot \mathbf{x}, \lambda_1) = - 2w_\eta + w_\Lambda .
    \end{align*}
To solve for $\max \mu^L(g \cdot \mathbf{x}, \lambda_r) = \max F_i(\gamma, w_1 , \dots , w_n) = 0$ notice that $F_8 = -2 F_{1}$. Since these functions are both non-positive they must both be zero. This gives the wall $w_\eta = \frac{1}{2}w_\Lambda$. The other strict semi-stability conditions, $w_{\eta'} \leq \frac{1}{3} (w_{\eta} + w_\Lambda)$ and $w_J \leq \frac{1}{3} (w_\eta + w_\Lambda) + \gamma$ are attained by evaluating the remaining inequalities $F_i \leq 0$ at $w_\eta = \frac{1}{2}w_\Lambda$.
\end{itemize}

\end{proof}

\begin{lemma}[Case of $C(2A_1)$]\label{Lemma:wallC(2A_1)}
    Let $\mathbf{x} := (C, p_1, \dots , p_n)$ have positive dimensional stabilizer, where $C$ is the union of a smooth conic and a transverse line. Then $\mathbf{x}$ is unstable for all line bundles $L$.
\end{lemma}

\begin{proof}
By Lemma \ref{lemma:ConicTransStbz} we know the marked points $p_i$ are at one of the two nodes of $C$ where the conic intersects the line. Denote the two nodes $\eta$ and $\eta'$. There is a matrix $g \in SL(3)$ which takes $\eta$ to $(1:0:0)$ and the line $\overline{\eta \eta'}$ to $\{z=0\}$. Then the maximal support of $g \cdot C$ is $\{xyz\}$ because $g \cdot C$ has a nodal singularity at $(1:0:0)$ with the line $\{z=0\}$ in the tangent cone. This gives

$$\mu^L(g \cdot \mathbf{x}, \lambda_r) = \sum_{p_i = \eta} w_i + r (\sum_{p_j = \eta'} w_j).$$

This is positive for $r=1$, so $\mathbf{x}$ is unstable.

 \end{proof}

Next we consider the cuspidal curve. If $(C,p_1, \dots , p_n)$ is a marked plane cubic with positive dimensional stabilizer and $C$ has a cusp, or $A_2$ singularity, then Lemma \ref{lemma:CuspStbz} implies that the marked points $p_i$ must either be at the cuspidal singularity or at the unique smooth inflection point of $C$.

\begin{lemma}[Case of $C(A_2)$]\label{Lemma:wallC(A_2)}
    Let $\mathbf{x}:=(C, p_1, \dots , p_n)$ have positive dimensional stabilizer, where $C$ is a cuspidal plane cubic. Define $I$ to be the subset of $\{1, \dots , n\}$ such that $p_i$ coincides with the cuspidal singularity of $C$ for $i \in I$. The complement $I^c$ indexes points $p_j$ coinciding with the smooth flex point. Then $\mathbf{x}$ is $L$ strictly semi-stable if and only if $\gamma= \frac{4}{3} \sum_{j \notin I} w_j - \frac{5}{3} \sum_{i \in I} w_i$ and $\sum_{i \in I} w_i \leq \frac{1}{2}\sum_{j \notin I} w_j$.
\end{lemma}

\begin{proof}
 We follow the same method as in Lemma \ref{lemma:Wall_3distinctlines} by solving $\max \{ \mu^L(g \cdot \mathbf{x}, \lambda_r) \} = 0$ for $L$, though we proceed more concisely. We mention two relevant geometric facts. First, if $f(x,y,z)$ is a homogeneous cubic polynomial and $V(f)$ contains a cusp at $(1:0:0)$ then the tangent cone at $(1:0:0)$ is a line of the form $\{ ay+bz=0 \}$. This, together with the Jacobian Criterion implies that $f(x,y,z) = x(ay+bz)^2 + h_3(y,z)$ for some homogeneous cubic $h$. Conversely, if $f$ has this form then $V(f)$ has a non-nodal singularity at $(1:0:0)$. Secondly we note that a curve $f(x,y,z)$ has a flex point at $(1:0:0)$ with flex line $\{z=0\}$ if and only if $f(x,y,z) = ay^3 + zh_2(x,y,z)$ for some $a \neq 0$ and $h_2$ homogeneous of degree $2$.

 Assume $\mathbf{x}$ is $L$ semi-stable. Let $p_c$ be the cuspidal singular point of $C$ and $p_f$ be the smooth flex point. If $g \in SL(3)$ does not take either $p_c$ or $p_f$ to $(1:0:0)$ then $\mu^L(g \cdot \mathbf{x}, \lambda_r)$ is negative for all $L$ in $\Lambda(\mc{C}_{n,d})$ and $r \in [\frac{-1}{2},1]$. Therefore, if $\mu^L(g \cdot \mathbf{x}, \lambda_r)=0$ then $g$ must take either $p_c$ or $p_f$ to $(1:0:0)$. There are four subsets of $SL(3)$ which maximize $\Xi_{max}(g \cdot \mathbf{x}).$ If $g_1$ takes $p_c$ to $(1:0:0)$ and the tangent cone at $p_c$ to the line $\{z=0\}$ then $\Xi_{max}(g_1 \cdot C) = \{y^3, xz^2\}$. If $g_2$ takes $p_f$ to $(1:0:0)$ and the flex line to $\{z=0\}$ then $\Xi_{max}(g_2 \cdot C) = \{y^3, x^2,z\}$. If $g_3$ takes $p_c$ to $(1:0:0)$ and the line $\overline{p_c p_f}$ to $\{z=0\}$ then $\Xi_{max}(g_3 \cdot C) = \{xy^2\}$. And if $g_4$ takes $p_f$ to $(1:0:0)$ and the line $\overline{p_c p_f}$ to $\{z=0\}$ then $\Xi_{max}(g_4 \cdot C) = \{x^2y\}$. 

 We abbreviate $\sum_{p_i = p_c} w_i$ as $w_c$ and $\sum_{p_j = p_f} w_j$ as $w_f$. Then the maximal values of $\mu^L(g \cdot \mathbf{x}, \lambda_r)$ are

 \begin{itemize}
     \item $\mu^L(g_1 \cdot \mathbf{x}, \lambda_r) = w_c + (-1-r)w_f - \gamma \max \{ 3r,-1-2r\},$ which is maximized in $r$ by
     \begin{align*}
        F_1 := \mu^L(g_1 \cdot \mathbf{x}, \lambda_{\frac{-1}{5}}) = w_c - \frac{4}{5}w_f + \frac{3}{5} \gamma
        && \text{and} &&
        F_2 := \mu^L(g_1 \cdot \mathbf{x}, \lambda_{\frac{-1}{2}}) = w_c - \frac{1}{2} w_f
     \end{align*}

     \item $\mu^L(g_2 \cdot \mathbf{x}, \lambda_r) = w_f + (-1-r)w_c - \gamma \max \{3r, 1-r \}$
     which is maximized in $r$ by
     \begin{align*}
        F_3 := \mu^L(g_1 \cdot \mathbf{x}, \lambda_{\frac{1}{4}}) = w_f - \frac{5}{4}w_c - \frac{3}{4} \gamma 
        && \text{and} &&
        F_4 := \mu^L(g_1 \cdot \mathbf{x}, \lambda_{\frac{-1}{2}}) = w_f - \frac{1}{2} w_c - \frac{3}{2} \gamma.
     \end{align*}

     \item $\mu^L(g_3 \cdot \mathbf{x}, \lambda_r) = w_c + r w_f - \gamma (1+2r)$
     which is maximized in $r$ by
     \begin{align*}
        F_5 := \mu^L(g_1 \cdot \mathbf{x}, \lambda_{\frac{-1}{2}})= w_c - \frac{1}{2} w_f 
        && \text{and} &&
        F_6 := \mu^L(g_1 \cdot \mathbf{x}, \lambda_1) = w_c + w_f - 3 \gamma.
     \end{align*}

     \item $\mu^L(g_4 \cdot \mathbf{x}, \lambda_r) = w_f + r w_c - \gamma(2+r) $
     which is maximized in $r$ by
     \begin{align*}
        F_7 := \mu^L(g_1 \cdot \mathbf{x}, \lambda_{\frac{-1}{2}}) = w_f - \frac{1}{2}w_c - \frac{3}{2} \gamma 
        && \text{and} &&
        F_8 := \mu^L(g_1 \cdot \mathbf{x}, \lambda_1) = w_f + w_c - 3\gamma .
     \end{align*}
 \end{itemize}
 We solve $\max{F_i}= 0$ for $(\gamma, w_1, \dots, w_n).$ Since $F_1 = \frac{-4}{5}F_3$ they must both be zero. This gives the wall $\gamma = \frac{4}{3}w_f - \frac{5}{3}w_c$, and the remaining inequalities evaluated at this wall become $w_c \leq \frac{1}{2}w_f.$

\end{proof}

In the case of the union of a conic and a tangent line with positive dimensional stabilizer, Lemma \ref{lemma:ConicTangStbz} tells us that the marked points are either at the intersection of the conic and the tangent line, or on one other line tangent to the conic.

\begin{lemma}[Case of $C(A_3)$] \label{Lemma:wallC(A_3)}
    Let $\mathbf{x} := (C, p_1, \dots , p_n)$ have positive dimensional stabilizer, where $C$ is the union of a conic $A$ and a tangent line $B$. Let $I$ be the subset of $\{1, \dots ,n\}$ such that $p_i$ coincides with the tacnode for $i \in I$ and let $J$ be the subset such that the $p_j$ coincide at a non-singular point of $B$ for $j \in J$. The remaining marked points coincide with the point on the conic $A$ whose tangent line intersects the $p_j$s. Then $\mathbf{x}$ is $L$ strictly semi-stable if and only if $\sum_{k \notin I \cup J} w_k = \sum_{i \in I} w_i + \gamma$ and $\sum_{i \in I} w_i - \gamma \leq \sum_{j \in J} w_j \leq \sum_{i \in I} w_i + 2 \gamma$.
\end{lemma}

\begin{proof}
First we remark that if all of the marked points lie on the line $B$ then $\mathbf{x}$ is unstable for all line bundles $L$. This can be seen by letting $g \in SL(3)$ take the tacnode to $(1:0:0)$ and the line $B$ to $\{z=0\}$. Then $\mu^L(g \cdot \mathbf{x}, \lambda_1) = \sum_1^n w_i$, which is positive.

Now assume that the marked points lie either at the tacnode or on the intersection $C \cap \Lambda$ where $\Lambda$ is a line \textit{other than $B$} which is tangent to the conic. We proceed as in Lemma \ref{lemma:Wall_3distinctlines} by solving $\max \{ \mu^L(g \cdot \mathbf{x}, \lambda_r) \ | \ g \in SL(3), \ r \in [\frac{-1}{2}, 1] \} = 0$ for $L$. There are  six subsets of $SL(3)$ which maximize $\Xi_{max}$. 

If $g\in SL(3)$ maximizes $\Xi_{max}(g \cdot \mathbf{x})$ then it must take the tacnode or one of other the marked points to $(1:0:0)$. Let $p_t$ be the tacnode, $p_a$ the marked point on the conic, $p_b$ the marked point on the linear component of $C$, and $\Lambda$ the line through $p_a$ and $p_b$. If $g_1$ takes $p_t$ to $(1:0:0)$ and $B$ to $\{z=0\}$ then $\Xi_{max}( g \cdot C) = \{xz^2, y^2z\}$. If $g_2$ takes $p_b$ to $(1:0:0)$ and $B$ to $\{z=0\}$ then $\Xi_{max} ( g_2 \cdot \mathbf{x}) = \{x^2z,  xy^2 \}$. If $g_3$ takes $p_b$ to $(1:0:0)$ and $\Lambda$ to $\{z=0\}$ then $\Xi_{max} ( g_3 \cdot \mathbf{x}) = \{x^2y\}$. If $g_4$ takes  $p_a$ to $(1:0:0)$ and $\Lambda$ to $\{z=0\}$ then $\Xi_{max} ( g_4 \cdot \mathbf{x}) = \{ x^2z, xy^2\}$. If $g_5$ takes $p_t$ to $(1:0:0)$ and the line $\overline{p_a p_t}$ to $\{z=0\}$ then $\Xi_{max} ( g_5 \cdot \mathbf{x}) = \{ xy^2\}$. Finally, if $g_6$ takes $p_a$ to $(1:0:0)$ and the line $\overline{p_a p_t}$ to $\{z=0\}$ then $\Xi_{max} ( g_6 \cdot \mathbf{x}) = \{x^2y\}$.

We abbreviate $\sum_{p_i = p_t} w_i$ as $w_t$, $\sum_{p_j = p_b} w_j$ as $w_b$, and $\sum_{p_k = p_a} w_k$ as $w_a$. Then the maximal values of $\mu^L( g \cdot \mathbf{x}, \lambda_r)$ are

\begin{itemize}
    \item $\mu^L(g_1 \cdot \mathbf{x}, \lambda_r) = w_t + rw_b +(-1-r)w_a - \gamma \max\{ -2-r, -1+r \}$, which is maximized in $r$ by
     \begin{align*}
        F_1 := \mu^L(g_1 \cdot \mathbf{x}, \lambda_1) = p_t -2p_a +p_b
        && \text{and} &&
        F_2 := \mu^L(g_1 \cdot \mathbf{x}, \lambda_0) = p_t - p_a + \gamma \end{align*}
        \begin{align*} \text{and} \ \ F_3 := \mu^L(g_1 \cdot \mathbf{x}, \lambda_{\frac{-1}{2}}) = p_t - \frac{p_a}{2} - \frac{p_b}{2}.   \end{align*}

    \item $\mu^L(g_2 \cdot \mathbf{x}, \lambda_r) = w_b + rw_t +(-1-r) w_a - \gamma \max\{1-r, 1+2r \}$, which is maximized in $r$ by
      \begin{align*}
        F_4 := \mu^L(g_2 \cdot \mathbf{x},  \lambda_1) = p_b -2 p_a + p_t + 3 \gamma 
        && \text{and} &&
        F_5 := \mu^L(g_2 \cdot \mathbf{x}, \lambda_{\frac{-1}{2}}) = p_b - \frac{1}{2}p_a - \frac{1}{2}p_t
    \end{align*}
    
    \item $\mu^L(g_3 \cdot \mathbf{x}, \lambda_r) = w_b + rw_a + (-1-r)w_t - \gamma (2+r)$, which is maximized in $r$ by
     \begin{align*}
        F_6 := \mu^L(g_3 \cdot \mathbf{x}, \lambda_1) = p_b -2 p_a + p_t + 3 \gamma
        && \text{and} &&
        F_7 := \mu^L(g_3 \cdot \mathbf{x}, \lambda_{\frac{-1}{2}}) = p_b - \frac{p_a}{2} - \frac{p_t}{2}.
    \end{align*}
    
    \item $\mu^L(g_4 \cdot \mathbf{x}, \lambda_r) = w_a + r w_b + (-1-r) w_t - \gamma \max \{1-r, 1+2r\}$, which is maximized in $r$ by
    \begin{align*}
        F_8 := \mu^L(g_1 \cdot \mathbf{x}, \lambda_1) = p_a + p_b - 2 p_t - 3 \gamma
        && \text{and} &&
        F_9 := \mu^L(g_1 \cdot \mathbf{x},  \lambda_0) = p_a - p_t - \gamma
        \end{align*}
        \begin{align*} \text{and} \ \
        F_{10} := \mu^L(g_1 \cdot \mathbf{x}, \lambda_{\frac{-1}{2}}) = p_a - \frac{p_b}{2} - \frac{p_t}{2} - \frac{3\gamma}{2}.
    \end{align*}

    \item $\mu^L(g_5 \cdot \mathbf{x}, \lambda_r) = w_t + rw_a + (-1-r)w_b - \gamma(1+2r)$, which is maximized in $r$ by
    \begin{align*}
        F_{11} := \mu^L(g_5 \cdot \mathbf{x}, \lambda_1) = w_t + w_a - 2w_b -3\gamma
        && \text{and} &&
        F_{12} := \mu^L(g_5 \cdot \mathbf{x}, \lambda_{\frac{-1}{2}}) = w_t - \frac{-1}{2}w_a - \frac{1}{2}w_b
    \end{align*}
    
    \item $\mu^L(g_6 \cdot \mathbf{x}, \lambda_r) = w_a + rw_t + (-1-r)w_b - \gamma(2+r)$, which is maximized in $r$ by
    \begin{align*}
        F_{13} := \mu^L(g_6 \cdot \mathbf{x}, \lambda_1) = w_a -2w_b + w_t - 3\gamma
        && \text{and} &&
        F_{14} := \mu^L(g_6 \cdot \mathbf{x}, \lambda_{\frac{-1}{2}}) = w_a - \frac{1}{2}w_b - \frac{1}{2}w_t - \frac{3}{2}\gamma.
    \end{align*}
\end{itemize}
We solve $\max \{F_i\} =0$ for $(\gamma, w_1, \dots , w_n)$. Since $F_2 = -F_9,$ they must both be zero. This gives the wall $p_a = p_t + \gamma.$ Evaluating the remaining inequalities at this wall gives the condition $w_t - \gamma \leq w_b \leq w_t +2 \gamma.$

\end{proof}

\begin{lemma}[Case of $C(D_4)$]\label{Lemma:wallD_4}
    Let $\mathbf{x} := (C, p_1, \dots , p_n)$ have positive dimensional stabilizer, where $C$ is the cone over three points. Let $I$ be the subset of $\{1, \dots , n\}$ such that $p_i$ coincides with the singular point of $C$ for $i \in I$. Then $\mathbf{x}$ is $L$ strictly semi-stable if and only if $\sum_{j \notin I} w_j = 3 \gamma + 2\sum_{i \in I} w_i$ and $\sum_{w_j = p} w_j \leq 2 \gamma + \sum_{i \in I}w_i$ for each non-singular points $p \in C$.
\end{lemma}

 \begin{proof}
As in the above Lemmas, we solve $\max \{\mu^L(g \cdot \mathbf{x}, \lambda_r) \ | \ g \in SL(3), r \in [\frac{-1}{2},1] \} = 0$ for $L$. From Lemma \ref{lemma:D4Stbz} we know that the marked points coincide either at the $D_4$ singular point or on some line. However, they do not all lie on the same linear component of $C$ because then $\mathbf{x}$ would be unstable. Let $\alpha$ be the non-singular point of $C$ at which the most weight is concentrated and let $\sigma$ be the singular point. To find the maximal values of $\Xi_{max}(g \cdot \mathbf{x})$ we may assume $g$ takes $\alpha$ or $\sigma$ to $(1:0:0)$. Then there are three subsets of $SL(3)$ which maximize $\Xi_{max}(g \cdot \mathbf{x})$. If $g_1$ takes $\sigma$ to $(1:0:0)$ and the line $\overline{\sigma \alpha}$ to $\{z=0\}$ then the equation for $g_1 \cdot C$ factors as three linear terms, one of which is $z$, and none of which are supported on the monomial $x$, so $\Xi_{max}(g_1 \cdot C) = \{ y^2z \}$. If $g_2$ takes $\alpha$ to $(1:0:0)$ and the line $\overline{\sigma \alpha}$ to $\{z=0\}$ then smoothness and the tangent direction at $(1:0:0)$ imply $\Xi_{max}(g_2 \cdot C) = \{x^2z\}$. If $g_3$ takes $\alpha$ to $(1:0:0)$ and the line through the other non-singular marked points to $\{z=0\}$ then $\Xi_{max}(g_3 \cdot C) = \{ x^2y\}$.

We abbreviate $\sum_{p_i = \sigma} w_i$ as $w_\sigma$, $\sum_{p_j = \alpha} w_j$ as $w_\alpha$, and $\sum_{p_k \neq \sigma, \alpha} w_k$ as $w_\beta$. Then the maximal values of $\mu^L(g \cdot \mathbf{x}, \lambda_r)$ are

\begin{itemize}
    \item $\mu^L(g_1 \cdot \mathbf{x}, \lambda_r) = w_\sigma + rw_\alpha + (-1-r)(w_\beta) - \gamma(-1+r)$,
    which is maximized in $r$ by
    \begin{align*}
        F_1 = \mu^L(g_1 \cdot \mathbf{x}, \lambda_{\frac{-1}{2}}) = w_\sigma - \frac{1}{2} w_\beta - \frac{1}{2}w_\alpha + \frac{3}{2} \gamma 
        && \text{and} &&
        F_2 = \mu^L(g_1 \cdot \mathbf{x}, \lambda_1) = w_\sigma + w_\alpha - 2w_\beta.
    \end{align*}
    
    \item $\mu^L(g_2 \cdot \mathbf{x}, \lambda_r) = w_\alpha + rw_\sigma + (-1-r)w_\beta - \gamma(1-r), $
    which is maximized in $r$ by
    \begin{align*}
        F_3 = \mu^L(g_2 \cdot \mathbf{x}, \lambda_{\frac{-1}{2}}) = w_\alpha - \frac{1}{2}w_\beta - \frac{1}{2}w_\sigma - \frac{3}{2} \gamma
        && \text{and} &&
        F_4 = \mu^L(g_2 \cdot \mathbf{x}, \lambda_1) = w_\sigma + w_\alpha - 2 w_\beta.
    \end{align*}
    
    \item$\mu^L(g_3 \cdot \mathbf{x}, \lambda_r) = w_\alpha + r w_\beta + (-1-r)w_\sigma - \gamma(2+r), $
    which is maximized in $r$ by
    \begin{align*}
        F_5 = \mu^L(g_3 \cdot \mathbf{x}, \lambda_{\frac{-1}{2}}) = w_\alpha - \frac{1}{2}w_\beta - \frac{1}{2}w_\sigma - \frac{3}{2}\gamma
        && \text{and} &&
        F_6 = \mu^L(g_3 \cdot \mathbf{x}, \lambda_1) = w_\alpha + w_\beta -2w_\sigma -3\gamma.
    \end{align*}
\end{itemize}
We solve $\max \{F_i\} =0$ for $(\gamma, w_1, \dots , w_n)$. Notice that $F_6 = -2F_1,$ so they both must be zero. This gives the wall $w_\gamma = \frac{1}{2}(w_\alpha + w_\beta) - \frac{3}{2} \gamma.$ Evaluating the remaining inequalities at this wall gives the condition $w_\alpha \leq 2\gamma + w_\sigma.$

\end{proof}

\section{Proof of the main theorems}

We now prove Theorem \ref{thm:GIT_walls} by combining the Lemmas from Section \ref{sec:innerwalls}.

\begin{thm} \label{thm:GIT_walls}
    For degree $d=3$, there are four types of inner walls, described below. The walls are segments of hyperplanes defined in Table \ref{table:GITwalls}. If the hyperplane segment intersects the interior of $\Lambda(\mc{C}_{n,d})$ then it is a GIT wall and all inner GIT walls are of this form.
    \begin{itemize}
        \item[(i)] For each nonempty, proper subset $I \subset [n]$ there is a hyperplane segment $W(3A_1, I)$ associated to the union of three non-concurrent lines $C(3A_1)$ with points $p_i, i \in I, $ supported at a node.  
        \item[(ii)] For each proper subset $I \subset [n]$ there is a hyperplane segment $W(A_2,I)$ associated to the cuspidal curve $C(A_2)$ with points $p_i, i \in I, $ supported at the cusp.
        \item[(iii)]  For each ordered pair of disjoint subsets $I,J \subset [n]$ there is a hyperplane segment $W(A_3, I, J)$ associated to  the union of a conic with a tangent line $C(A_3)$ with points $p_i, i \in I, $ supported at the tacnode and $p_j, j \in J, $ supported at the unique linear component of the curve.
        \item[(iv)] For each subset $I \subset [n]$ with $|I| \leq n-3$ there is a hyperplane segment $W(D_4, I)$ associated to the cone over three points $C(D_4)$ with points $p_i, i \in I, $ supported at the singularity. 
    \end{itemize}
\end{thm}

\begin{proof}

Since the Hilbert-Mumford function is continuous, stability only changes at line bundles for which there exist strictly semi-stable marked curves. Due to Lemma \ref{lemma:stab}, the GIT walls are the locus of equivalence classes of line bundles $L$ in $NS^{SL(3)}(\mc{C}_{n,d})_\Q$ for which there exists a marked curve with positive-dimensional $SL(3)$ stabilizer which is $L$ strictly semi-stable. If $(C,p_1, \dots , p_n)$ has positive-dimensional stabilizer, then the curve $C$ must also have positive dimensional stabilizer. If $C$ is a reduced cubic plane curve, then by Lemma \ref{lemma:CubicPositStbz} there are only five such curves, up to projective equivalence. Suppose $\mathbf{x}:=(C, p_1, \dots , p_n)$ is a marked cubic with positive dimensional stabilizer and $L$ is an ample line bundle in $NS(\mc{C}_{n,3})$ at which $\mathbf{x}$ is strictly semi-stable.

\begin{enumerate}[label=(\roman*)]
    \item If $C$ is three non-concurrent lines then by Lemma \ref{lemma:TriangleStbz} the marked points lie either on some nodal singular point $\eta$ of $C$ or on the opposite line $\Lambda$. Let $I$ be a nonempty, proper subset of $\{1, \dots , n\}$. If the points $p_i$ coincide with $\eta$ for $i \in I$ and $p_j$ are distinct nonsingular points of $\Lambda$ for $j \notin I$ then Lemma \ref{lemma:Wall_3distinctlines} implies that $(C, p_1, \dots, p_n)$ is $L$ strictly semi-stable if and only if $L$ lies on the wall $\sum_{i \in I} w_i = \frac{1}{2} \sum_{j \notin I} w_j$ bounded by $ w_m \leq  \frac{1}{2} \sum_{j \notin I} w_j + \gamma$ for each $ m \notin I.$ We remark that if the cardinality of $I$ is empty or all of $\{1, \dots n\}$ then the stability conditions of Lemma \ref{lemma:Wall_3distinctlines} cannot be satisfied.

    \item If $C$ is a cuspidal cubic then by Lemma \ref{lemma:CuspStbz} the marked points coincide either with the cusp of $C$ or the unique smooth flex point. Let $I$ be a subset of $\{1, \dots , n\}$. If the points $p_i$ coincide with the cusp for $i \in I$ and the points $p_j$ coincide with the smooth flex point then Lemma \ref{Lemma:wallC(A_2)} implies $(C, p_1, \dots , p_n)$ is $L$ strictly semi-stable if and only if $L$ lies on the wall $\sum_{i \in I} w_i  = \frac{4}{5} \sum_{j \notin I} w_j  - \frac{3}{5}\gamma $ bounded by $ \sum_{i \in I} w_i \leq \frac{1}{2} \sum_{j \notin I} w_j.$ 

    \item If $C$ is the union of a conic and a tangent line then by Lemma \ref{lemma:ConicTangStbz} there exists a line $\Lambda$ tangent to the conic such that the marked points either coincide with the singular point of $C$, or lie on $\Lambda \cap C$. Let $I$ and $J$ be disjoint subsets of $\{1, \dots , n\}$. If the points $p_i$ coincide at the tacnode for $i \in I$ and $p_j$ coincide at a smooth point of the linear component of $C$ for $j \in J$ then Lemma \ref{Lemma:wallC(A_3)} implies that $(C, p_1, \dots , p_n)$ is $L$ strictly semi-stable if and only if $L$ lies on the wall $ \sum_{i \in I} w_i  = \frac{4}{5} \sum_{j \notin I} w_j  - \frac{3}{5}\gamma ,$ bounded by $ \sum_{i \in I} w_i \leq \frac{1}{2} \sum_{j \notin I} w_j$.

    \item If $C$ is the cone over three points then by Lemma \ref{lemma:D4Stbz} there exists a line $\Lambda$ such that the marked points are either at the singularity of $C$ or on the line $\Lambda$. Let $I$ be a subset of $\{1, \dots , n\}$ and let $A_1 \sqcup A_2 \sqcup A_3$ be a partition of the complement of $I$. If the points $p_i$ coincide at the $D_4$ singularity for $i \in I$, and the points $p_j$ lie on one of the first, second, or third linear components of $C$ for $j \in A_1, A_2,$ and $A_3,$ respectively, then Lemma \ref{Lemma:wallD_4} implies that  $(C, p_1, \dots, p_n)$ is $L$ strictly semi-stable if and only if $L$ lies on the wall $\sum_{i \in I} w_i = \frac{1}{2} \sum_{j \notin I} w_j - \frac{3}{2} \gamma$ bounded by $\sum_{j \in A_k} w_j  \leq \sum_{i \in I} w_i + 2 \gamma$. Note that this wall may stretch beyond this bound given by one particular partition $A_1, A_2, A_3$. As long as there exists some partition $A_1 \sqcup A_2 \sqcup A_3 = \{1, \dots , n\} \ \backslash \ I$,  such that $L$  satisfies $\sum_{i \in I} w_i = \frac{1}{2} \sum_{j \notin I} w_j - \frac{3}{2} \gamma$ and $\sum_{j \in A_k} w_j  \leq \sum_{i \in I} w_i + 2 \gamma$, then we can find a marked curve that is $L$ strictly semi-stable. 

    \item If $C$ is the union of a conic and a transverse line then Lemma \ref{lemma:ConicTransStbz} shows that the marked point must both be nodal points of $C$, but Lemma \ref{Lemma:wallC(2A_1)} shows that such a marked curve will be unstable for all $L$.
\end{enumerate}
    
\end{proof}

We now describe how the stability of marked cubic curves changes as these GIT walls are traversed.

\begin{thm}\label{thm:wallcrossing}
Let $S(T, I, -)$ and $S(T, I, +)$ be the generic marked curves that change stability when the wall $W(T,I)$ is crossed, with common degeneration $S(T,I,0)$ as described in Section \ref{sec:intro}. Then the wall crossing behavior is described as follow:

\begin{itemize}
\item[(Case i)] $S(3A_1,I,0)$ is a union of three non-concurrent lines. The marked points indexed by $I$ coincide at the nodal intersection of two lines and the remaining marked points lie on the third line.
\begin{itemize}
    \item[$\bullet$] $S(3A_1,I,-)$ is an irreducible nodal curve with the marked points indexed by $I$ coinciding at the  $A_1$ singularity and the remaining marked points in general position.
    \item[$\bullet$] $S(3A_1,I,+)$ is the union of a conic and a transverse line, with marked points $p_j$ lying in general position on the linear component for $j \notin I.$
\end{itemize}
\item[(Case ii)] $S(T,I,0)$ is a plane cubic with a cuspidal singularity. The marked points indexed by $I$ coincide at the $A_2$ singularity and the others coincide at the curves unique inflection point.
\begin{itemize}
    \item[$\bullet$] $S(A_2,I, -)$ is an irreducible cuspidal cubic curve with marked points indexed by $I$ coinciding at the $A_2$ singularity.
    \item[$\bullet$] $S(A_2,I,+)$ is a smooth cubic curve with marked points $p_j$ coinciding at an inflection point for $j \notin I$. 
\end{itemize}
\item[(Case iii)] $S(A_3,I,J,0)$ is the union of a conic and a line which intersect at a tacnode. The marked points indexed by $I$ coincide at the $A_3$ singularity, those indexed by $J$ lie on the linear component of the curve, and the rest lie on a single line in the plane which is also tangent to the conic.
    \begin{itemize}
      \item[$\bullet$] $S(A_3, I, J,-)$ is the union of a conic with a tangent line with marked points indexed by $I$ coinciding at the $A_3$ singularity and those indexed by $J$ lying on the tangent line.
     \item[$\bullet$] $S(A_3,I,J,+)$ is a smooth cubic curve $C$ with marked points $p_k$ coinciding for $k \notin I \cup J$ and the marked points indexed by $J$ coinciding at the transversal intersection of $C$ with the line tangent to $C$ at $p_k$.
    \end{itemize}
    \item[(Case iv)] $S(D_4,I,0)$ is a plane cubic with a $D_4$ singularity. The marked points indexed by $I$ coincide at the $D_4$ singularity. The marked points indexed by $B_1$, $B_2$, and $B_3$ coincide at $3$ points on the $3$ linear component of the curve, respectively, and these $3$ points are collinear.
   \begin{itemize}
    \item[$\bullet$] $S(D_4,I, -)$ is three concurrent lines with marked points indexed by $I$ coinciding at the $D_4$ singularity and the points indexed by $B_1$, $B_2,$ and $B_3$ lying on the three lines, respectively.
    \item[$\bullet$] $S(D_4,I, +)$ is a smooth cubic curve with the points indexed by $B_1$ coinciding at a point $q_1$, the points indexed by $B_2$ coinciding at a different point $q_2$, and the points indexed by $B_3$ coinciding at a third point on the line $\overline{q_1q_2}$.
    \end{itemize}
\end{itemize}
See Figure \ref{fig:WallCrossing} for an illustration of the walls.
\end{thm}

\begin{proof}
    Each wall in Theorem \ref{thm:GIT_walls} was found by computing the maximal values of $\mu^L(g \cdot \mathbf{x}, \lambda_r)$ for some marked curve $\mathbf{x}$ with positive-dimensional stabilizer.  In each case (Lemmas \ref{lemma:Wall_3distinctlines} - \ref{Lemma:wallD_4}) there are some linear automorphisms, $g^+$ and $g^-$, and normalized one-parameter subgroups, $\lambda^+$ and $\lambda^-$, such that $\mu^L(g^- \cdot \mathbf{x}, \lambda^-)$ is a negative multiple of $ \mu^L(g^+ \cdot \mathbf{x}, \lambda^+)$. Furthermore, there may be several $g_a$ and $\lambda_A$ for which $\mu^L(g_a \cdot \mathbf{x}, \lambda_A)$ is a positive multiple of $\mu^L(g^+ \cdot \mathbf{x}, \lambda^+)$ and several $g_b$ and $\lambda_B$ for which $\mu^L(g_b \cdot \mathbf{x}, \lambda_B)$ is a positive multiple of $\mu^L(g^- \cdot \mathbf{x}, \lambda^-)$. These are the linear automorphisms and one-parameter subgroups which make $\mu^L(g \cdot \mathbf{x}, \lambda_r)$ zero at the wall associated to $\mathbf{x}$. All of the other maximal values of $\mu^L( g \cdot \mathbf{x}, \lambda_r)$ are negative on the interior of the wall and therefore negative at the adjacent line bundles $L_{(\pm)}$. Any marked curve $\mathbf{y}$ for which the set $\Xi_{max}(\mathbf{y})$ is equal to $ \Xi_{max}( g_a \cdot \mathbf{x})$ or $\Xi_{max}( g_b \cdot \mathbf{x})$ will have $\mu^L(\mathbf{y}, \lambda_A)=0$, or $\mu^L(\mathbf{y}, \lambda_B)=0$ respectively, at the associated wall.

    For each $\mathbf{x}$ as above we find the generic marked curve having maximal support equal to $ \Xi_{max}( g_a \cdot \mathbf{x})$ for at least one of the $g_a$ and the generic marked curve having the maximal support $\Xi_{max}( g_b \cdot \mathbf{x})$ for at least one of the $g_b$. That is, we find a marked curve $\mathbf{y}$ such that
    \begin{enumerate} 
        \item there exists an $h \in SL(3)$ such that $\Xi_{max}(h \cdot \mathbf{y}) = \Xi_{max}(g_a \cdot \mathbf{x})$ for some $g_a$,
        \item there does not exist any $h \in SL(3)$ such that $\Xi_{max}(h \cdot \mathbf{y}) = \Xi_{max}(g_b \cdot \mathbf{x})$ for any $g_b$,
        \item and for all $h' \in SL(3)$ there exists a $g' \in SL(3)$ such that $\Xi_{max} (h' \cdot \mathbf{y}) \leq \Xi_{max}(g' \cdot \mathbf{x})$.
    \end{enumerate}
     This implies $\mathbf{y}$ is unstable on the side of the wall where $\mu^L(g_a \cdot \mathbf{x}, \lambda_A)$ is positive and stable on the side where it is negative. Likewise, we find a marked curve $\mathbf{y'}$ that has the same properties with respect to $\Xi_{max}(g_b \cdot \mathbf{x}).$

     In Lemma \ref{lemma:Wall_3distinctlines} we found the wall corresponding to the marked cubic $\mathbf{x} = (C, p_1, \dots , p_n)$ with three $A_1$ singularities and positive-dimensional stabilizer. In that Lemma we found four subsets of $SL(3)$, corresponding to matrices $g_1,g_2,g_4,$ and $g_6$ for which $\mu^L(g_1 \cdot \mathbf{x}, \lambda_{\frac{-1}{2}})$ is a positive multiple $\mu^L(g_2 \cdot \mathbf{x}, \lambda_{\frac{-1}{2}})$ and a negative multiple of $\mu^L(g_4 \cdot \mathbf{x}, \lambda_1)$ and $\mu^L(g_6 \cdot \mathbf{x}, \lambda_1)$. These are the numerical functions whose vanishing defines the wall $W(A_3, I)$. The maximal support $\Xi_{max}(g_1 \cdot \mathbf{x})$ is a tuple of sets with components $\Xi_{max}(g_1 \cdot C) = \{xyz\}$, $\Xi_{min}(g_1 \cdot p_i) = \{x\}$ for $i \in I$, $\Xi_{min}(g_1 \cdot p_j) = \{y\}$ for $p_j$ at the node $\eta'$, and $\Xi_{min}(g_1 \cdot p_k) = \{z\}$ for the remaining points. The maximal support $\Xi_{max}(g_2 \cdot \mathbf{x})$ contains components $\Xi_{max}(g_2 \cdot C) = \{xy^2\}$, $\Xi_{min}(g_2 \cdot p_i) = \{x\}$ for $i \in I$, $\Xi_{min}(g_2 \cdot p_j) = \{y\}$ for $p_j$ the smooth point of $\Lambda$ at which the most weight is concentrated, and $\Xi_{min}(g_2 \cdot p_k) = \{z\}$ for the remaining points. The maximal support $\Xi_{max}(g_4 \cdot \mathbf{x})$ contains components $\Xi_{max}(g_4 \cdot C) = \{xyz\}$, $\Xi_{min}(g_4 \cdot p_j) = \{x\}$ for $p_j$ at the node $\eta'$, $\Xi_{min}(g_4 \cdot p_k) = \{y\}$ for the remaining points on $\Lambda$, and $\Xi_{min}(g_4 \cdot p_i) = \{z\}$ for $i \in I$. Finally the maximal support $\Xi_{max}(g_6 \cdot \mathbf{x})$ contains components $\Xi_{max}(g_6 \cdot C) = \{x^2z\}$, $\Xi_{min}(g_6 \cdot p_j) = \{x\}$ for $p_j$ the smooth point of $\Lambda$ at which the most weight is concentrated, $\Xi_{min}(g_6 \cdot p_k) = \{y\}$ for the remaining points on $\Lambda$, and $\Xi_{min}(g_6 \cdot p_i) = \{z\}$ for $i \in I$. Any marked curve having maximal support equal to $\Xi_{max}(g_1 \cdot \mathbf{x})$ or $\Xi_{max}(g_2 \cdot \mathbf{x})$ will therefore have a nodal singularity coinciding with the points $p_i$ for $i \in I$. On the other hand, any marked curve having maximal support equal to $\Xi_{max}(g_4 \cdot \mathbf{x})$ or $\Xi_{max}(g_6 \cdot \mathbf{x})$ will contain a linear component which is incident with the points $p_j$ for $j \notin I$. Therefore the generic marked curves satisfying the conditions in the above paragraph are an irreducible nodal cubic with marked points $p_i$ at the node for $i \in I$ and the union of a smooth conic and a transverse line with smooth marked points and $p_j$ on the linear component for $j \notin I$.

     The remaining wall crossing behaviors of Theorem \ref{thm:wallcrossing} are found similarly. This shows that the marked curves $S(T,I, \pm)$ in the theorem change stability at the wall associated to $T$. To see that they are the general curves that do so we count dimensions. That is, we verify that at each wall the dimension of the set of curves $S(T,I,-)$ in $\overline{M}^{git}_{1,L^-}$ plus the dimension of $S(T,I,+)$ in $\overline{M}^{git}_{1,L^+}$ is equal to the codimension of $S(T,I,0) $ in  $\overline{M}^{git}_{1,L^0}$, minus two. Since the exceptional loci of the wall crossing morphism $\overline{M}^{git}_{1,L^-} \dashrightarrow \overline{M}^{git}_{1,L^+}$ are obtained as weighted projective space bundles over the positive and negative weight subspaces of the normal bundle of $S(T,I,0) $, this dimension count shows that the curves $S(T,I, \pm)$ are dense in the exceptional loci.

\end{proof}

\section{Applications to other moduli spaces} \label{sec:app}

\subsection{Cubics with  two points and cubic surfaces} \label{sec:2pts}

We now specialize to the case of curves of degree $d=3$ with $n=2$ marked points. In Corollary \ref{cor:FromM2toRadu} we find a GIT chamber containing a line bundle $L$ such that $\mc{C}_{2,3}^s \slash_{L} SL(3)$ maps into the moduli space of cubic surfaces.

In the case of $n=2$ points and arbitrary degree $d$, $\Lambda(\mc{C}_{n,d})$ is the cone over a pentagon (Corollary \ref{cor:pentagon}). Since GIT stability remains invariant under taking tensor powers of a line bundle, we assume $\gamma = 1$ to obtain a hyperplane section of the $\Lambda(\mc{C}_{n,d})$, called the \textit{linearization polytope} $\Delta(\mc{C}_{2,d})$. In this section we use $p_i$ and $p_j$ to refer to the points $p_1$ and $p_2$ with the understanding that $i \neq j$.

\begin{cor}\label{cor:pentagon}
    The linearization polytope $\Delta(\mc{C}_{2,d})$ is the pentagon in the hyperplane $\{\gamma =1 \} \subset NS^{SL(3)}(\mc{C}_{2,d})_\Q$ given by the inequalities
$$
\{
(1, w_1, w_2) \in 
NS^{SL(3)}(\mc{C}_{2,d})_\Q
\; \big| \;
w_i \leq w_j + d -2, \ w_1 + w_2 \leq d, \ w_i \geq 0   \}.
$$
\end{cor}
\begin{proof}
We apply Theorem \ref{thm:ample_cone} to the case of $n=2$ points and normalize with the hyperplane $\gamma=1$. This gives the inequalities 
\begin{enumerate}[label=(\Alph*)]
    \item $w_i \leq \frac{1}{2}w_j + d - \frac{3}{2}$
    \item $ w_i \leq w_j +d-2$
    \item $w_i + w_j \leq d$
    \item and $w_i \geq 0.$
\end{enumerate}
However, the first inequality is redundant: Adding (B) and (C) gives $2w_i + w_j \leq w_j + 2d -2$. This simplifies to $w_i \leq d-1$. Adding (B) to this inequality gives $2w_i \leq w_j +2d -3$, which implies (A).

\end{proof}

 We apply Theorem \ref{thm:GIT_walls} to the case of curves of degree $3$ with $2$ points and $\gamma = 1$. We list all the walls that lie in the interior of $\Delta(\mc{C}_{2,3}).$

\begin{cor} \label{cor:2pts}
The wall crossing behavior for a marked cubic curve $(C,p_1, p_2)$ with respect to a line bundle with parameters $(\gamma, w_1, w_2)$ are listed in the following table. Recall that Table \ref{table:GITwalls} describes the GIT walls as hyperplane segments which intersect $\Lambda(\mc{C}_{n,d})$.
The first two columns in the table below lists these walls for the case of $n=2, \ d=3$ with notation as in Table \ref{table:GITwalls}. The third column describes the marked cubics that become unstable when the left-hand side of the equation in column 2 exceeds the right-hand side. The fourth column described the marked cubics that become unstable when the right-hand side is larger. \\

 \begin{tabular}{
 |p{2cm}| p{3cm}|p{5cm}|p{5cm}|
 }
\hline 
Inner Walls & Equation &  Unstable when LHS $>$ RHS & Unstable when LHS $<$ RHS \\
\hline 

$W(3A_1, \{i\})$ & $w_i = \frac{1}{2} w_j$ & $p_i$ lies at an $A_1$ singularity. & $p_j$ lies on a linear component of $C$. \\

\hline
$W(A_2, \emptyset)$ &  $0  = 4( w_1 + w_2)  - 3$  & $C$ contains an $A_2$ singularity.  & $p_1=p_2$ coincide at an inflection point of $C$.  \\

\hline
$W(A_2, \{i\})$ &  $ w_i  = \frac{4}{5}w_j  - \frac{3}{5}$  & $p_i$ lies at an $A_2$ singularity.  & $p_j$ lies at an inflection point.  \\

\hline
$W(A_3, \emptyset, \emptyset)$ &  $0 =  w_1+w_2 - 1$ & $C$ contains an $A_3$ singularity.  &  $p_1=p_2$. \\

\hline
$W(A_3, \emptyset, \{i\})$ &  $0 =  w_j - 1$ & $p_j$ lies on a linear component of $C$ which is tangent to a conic component.  &  $p_j$ lies on the line which is tangent to $C$ at $p_i$. \\

\hline

\end{tabular}
    
\end{cor}

\begin{proof}
    We evaluate the wall crossings of Theorem \ref{thm:GIT_walls} in the case of $2$ points and $\gamma = 1$. We only list the walls, corresponding to subsets $I, J \subset \{1,2\}$, which lie in the interior of the linearization polytope. For instance, if $C$ has a $D_4$ singularity and $p_1$ and $p_2$ are nonsingular points of $C$ then $(C,p_1, p_2)$ is only semi-stable at the wall $W(D_4, \emptyset)$ given by $w_1 + w_2 = 3.$ But this is an outer wall, forming a facet of the pentagon $\Delta(\mc{C}_{2,3}).$
\end{proof}

 We recall Radu Laza's Thesis \cite{Laz06} in which he constructs a series of compactifications of the moduli space of \textit{degree $3$ pairs} consisting of a cubic curve and a line in $\Pj^2$  by taking the VGIT quotients of $\Pj(\Gamma(\Pj^2, \mc{O}(3))) \times \Pj(\Gamma(\Pj^2, \mc{O}(1))) $ by $SL(3)$, denoted $M^{1,3}_{pairs}(t)$ . He finds a VGIT chamber corresponding to the line bundle parameter $t= \frac{3}{2}- \epsilon$ such that $M^{1,3}_{pairs}(\frac{3}{2}- \epsilon)$ is the moduli space of pairs $(C,L)$ where $C$ has at worst isolated singularities of type $A_k$ and $L$ is a line intersecting $C$ transversely. Laza then considers the moduli spaces $M^{(1,3) lab}_{pairs}(t)$ of pairs of a plane cubic and a line, with labeled intersection. He uses a classical construction to prove that $M^{(1,3) lab}_{pairs}(\frac{3}{2}- \epsilon)$ is the coarse moduli space for cubic surfaces containing a marked Eckardt point and at worst $A_k$ singularities.
In our case, there is a chamber represented by $\mathbf{w}$ giving rise to a compact moduli space of plane cubics with two marked points $\overline{M}^{git}_{1, \mathbf{w}}$, which is isomorphic to $M^{(1,3) lab}_{pairs}(\frac{3}{2}- \epsilon)$ and hence isomorphic to $\Pj(1,2,2,3)$ \cite[Corollary 3.15]{Laz06}.

\begin{cor}\label{cor:FromM2toRadu}
    There exists an isomorphism
    \[
       \phi^{lab}:\mc{C}_{2,3}^s \slash_{L} SL(3) \rightarrow M^{(1,3)lab}_{pairs}\left(
       \frac{3}{2}- \epsilon
       \right) 
       \cong 
       \mathbb{P}(1,2,2,3)
    \]
    where $L$ is a line bundle corresponding to a vector $\mathbf{w}$ in the GIT chamber $\{ w_1 > 1, w_2 > 1, w_1 + w_2 < 3\} \subset \Delta(\mc{C}_{2,3})$.
\end{cor}

\begin{proof}[Proof of Corollary \ref{cor:FromM2toRadu}]
    First we construct a morphism $\psi: \mc{C}_{2,3}^{s}(L) \rightarrow \Pj(\Gamma(\Pj^2, \mc{O}(3))) \times \Pj(\Gamma(\Pj^2, \mc{O}(1)))$ by associating to a marked cubic $(C,p_1, p_2)$ the degree $3$ pair $(C, \overline{p_1p_2}).$ This can be done explicitly: if $p_1 = (x_1:y_1:z_1)$ and $p_2 = (x_2:y_2:z_2)$ then the line $\overline{p_1p_2}$ is given by the equation $\{ (x_2z_1 - x_1z_2)(z_1y-y_1z) = (y_2z_1 - y_1z_2)(z_1x - x_1z) \}$. Suppose $(C,p_1,p_2)$ is $L$ (semi-)stable. Since $w_1$ and $w_2$ are greater than $1$, the stability condition associated to the walls $W(A_3, \emptyset, \{i\})$ imply that neither point $p_i$ can lie on a line that is tangent to the curve $C$ at the other point $p_j$. We also know from the wall-crossing behavior associated to the walls $W(3A_1, \{i\})$ and $W(A_2, \{i\})$ that neither of the points $p_i$ may coincide with a singularity of $C$. Therefore, the $L$ semi-stability of $(C,p_1,p_2)$ implies that $C$ is a cubic curve only containing singularities of type $A_k$  with $k\in \{1,2, 3 \}$ and that the line $\overline{p_1p_2}$ always intersects $C$ transversely at three distinct, smooth points. Recalling Laza's GIT analysis \cite[Theorem 3.2]{LazaThesis}, this implies that the degree $3$ pair $(C, \overline{p_1p_2})$ is $\frac{3}{2} - \epsilon$ semi-stable.

    Since the image of $\psi$ is contained in $\big( \Pj(\Gamma(\Pj^2, \mc{O}(3))) \times \Pj(\Gamma(\Pj^2, \mc{O}(1))) \big)^{ss}(\frac{3}{2}- \epsilon), $ we can restrict to this codomain and then compose with the quotient map to get a morphism $\psi':  \mc{C}_{2,3}^{ss}(L) \rightarrow M^{1,3}_{pairs}(\frac{3}{2}- \epsilon).$ This map $\psi'$ is clearly $SL(3)$ invariant: if $(C, p_1,p_2) = g \cdot (D, q_1, q_2)$ then $(C,\overline{p_1p_2}) = g \cdot (D, \overline{p_1,p_2})$, so their image under $\psi'$ is the same orbit. Then, by the universal property of $\pi: \mc{C}_{2,3}^{s}(L) \rightarrow \mc{C}_{2,3}^s \slash_L SL(3)$ as a quotient, there is a morphism $\phi: \mc{C}_{2,3}^s \slash_L SL(3) \rightarrow  M^{1,3}_{pairs}(\frac{3}{2}- \epsilon)$ such that $\psi'$ factors as $ \phi \circ \pi.$

    The above map is 6:1 because for a generic degree $3$ pair $(C,L)$ there are $3 \cdot 2$ choices of ordered pairs of points in the set $C \cap L$ that map to $(C,L)$. However, we recall from Laza's construction  (Lemma 3.13 in \cite{LazaThesis}) that we can label the $3$ points of intersection in $C \cap L$ to obtain an isomorphism from the moduli space $M^{(1,3)lab}_{pairs}(\frac{3}{2}- \epsilon)\cong \mathbb{P}(1,2,2,3)$ of \textit{degree 3 pairs with labeled intersections} to the moduli space of cubic surfaces with a marked Eckardt point.
    
    We lift $\phi$ to $\phi^{lab}: \mc{C}_{2,3}^s \slash_L SL(3) \rightarrow M^{(1,3)lab}_{pairs}(\frac{3}{2}- \epsilon)$ by ordering the three intersections in $\phi((C,p_1, p_2))$ as $(p_1, p_2, q)$ where $q$ is the third intersection of $C$ with $\overline{p_1p_2}$. We observe that $\phi^{lab}$ is a birational morphism, bijective on closed points. Its target is the weighted projective space $ M^{(1,3)lab}_{pairs}(\frac{3}{2}- \epsilon) \cong \mathbb{P}(1,2,2,3)$, which is normal. A form of Zariski's Main Theorem \cite[Corollary 12.88]{Gortz10}
    then implies $\phi^{lab}$ is an isomorphism.
\end{proof}

\begin{rmk}
We do not expect the map $(C, p_1, p_2) \mapsto (C, \overline{p_1p_2})$ to extend when using other line bundles in the moduli of plane cubics and two points. Indeed, for every other pair of chambers $X$ and $Y$ in the linearization polytopes of $\mc{C}_{2,3}$ and degree $3$ pairs respectively, we can find a tuple $(C, p_1, p_2) $ which is stable with respect to $X$, but for which $(C, \overline{p_1p_2})$ is unstable with respect to $Y$.
\end{rmk}

\subsection{Plane Cubics and \texorpdfstring{$M_{1,n}$}{M1,n}} \label{sec:planecubicsM1n}
It is well known that a smooth cubic curve $ V (f) $ has nine 
inflection points, which can be obtained by intersecting the vanishing locus of $f$ and the vanishing locus of its Hessian determinant $He(f)$. To form the parameter space for cubic plane curves marked with $n$ 
points and an inflection point, we define a subvariety of the product of $\Pj(\Gamma(\mc{O}(3),\Pj^2))$ with $n+1$ copies of $\Pj^2$:
$$\mc{C}_{n,d}' :=\{
(f, p_1, \ldots, p_{n+1}) \in \mathbb{P}^9 \times (\mathbb{P}^2)^{n+1} \; | \;
f(p_i)=0, \; He(f)(p_{n+1)}) =0 \}.
$$

 $\mc{C}_{n,3}'$ is therefore a closed subscheme of $\mc{C}_{n+1,3}$ and by the Hilbert-Mumford numerical criterion, a marked curve $(C, p_1, \dots , p_n) \in \mc{C}_{n,3}'$ is $L$ (semi-)stable if and only if it is $L$ (semi-)stable in $\mc{C}_{n+1,3}$. We therefore have the following commuting square:
$$
\begin{tikzcd}
    \mc{C}_{n,3}^{' ss}(L) \arrow[r, hook] \arrow[d] & \mc{C}_{n+1,3}^{ss}(L) \arrow[d] \\
    \mc{C}_{n,3}' \sslash_L SL(3) \arrow[r, hook] & \mc{C}_{n+1,3} \sslash_L SL(3) .
\end{tikzcd}
$$

\begin{cor}\label{thm:M_1n}
    $M_{1,n+1}$ is isomorphic to an open subset of $\mc{C}_{n,3}^{'s} \slash_L SL(3)$ for some $SL(3)-$linearized line bundle $L$.
\end{cor}

\begin{proof}[Proof of Theorem \ref{thm:M_1n}]
Let $\mc{C}_{n,3}'^{\circ} \subset \mc{C}_{n,d}$ be the open locus of smooth plane cubics marked with $n$ points and an additional inflection point which are all distinct. For $0 < \epsilon \ll 1$, the entirety of $\mc{C}_{n,3}^{' \circ}$ is $(1, \epsilon, \dots , \epsilon)$ stable by Theorem \ref{thm:wallcrossing}. Let $\Pj^2 \times \mc{C}_{n,3}^{' \circ} \rightarrow \mc{C}_{n,3}^{' \circ} $ be the trivial $\Pj^2$ bundle. Let $B \hookrightarrow \Pj^2 \times \mc{C}_{n,3}^{' \circ}$ be the family of curves over $\mc{C}_{n,3}^{' \circ} $ whose fiber over $(C, p_1, \dots, p_{n+1})$ is the curve $C$ in the plane $\Pj^2 \times \{(C, p_1, \dots, p_{n+1})\} $. The family of genus $1$ curves $B$ and the $n+1$ global sections given by the distinct marked points define a map $\phi: \mc{C}_{n,3}^{' \circ} \rightarrow M_{1,n+1} $. If $(C, p_1, \dots, p_{n+1})$ and $(D, q_1, \dots, q_{n+1})$ are projectively equivalent then their images under $\phi$ are isomorphic marked curves. Therefore $\phi$ factors through the categorical quotient $\mc{C}_{n,3}^{' \circ} \slash_L SL(3)$.

$$
\begin{tikzcd}
    B \arrow[r, hook] \arrow[dr] & \Pj^2 \times \mc{C}_{n,3}^{' \circ} \arrow[d] & \\
    & \mc{C}_{n,3}^{' \circ}  \arrow[r, "\phi"] \arrow[d] & M_{1,n+1} \\
    & \mc{C}_{n,3}^{' \circ} \slash_L SL(3) \arrow[ur, "\overline{\phi}"] & 
\end{tikzcd}
$$

 We show that $\overline{\phi}$ is an isomorphism by showing that it is a surjective birational morphism and invoking Zariski's Main Theorem. Let $(E,a_1, \dots, a_{n+1}) \in M_{1,n+1}$ be a marked elliptic curve. The linear system $|3a_{n+1}|$ induces an embedding $\psi_{|3a_{n+1}|}$ of $(E,a_1, \dots, a_{n+1})$ into the plane as a marked cubic curve such that the image of $a_{n+1}$ is an inflection point and $\overline{\phi}(\psi_{|3a_{n+1}|}(E),\psi_{|3a_{n+1}|}(a_1), \dots, \psi_{|3a_{n+1}|}(a_{n+1})) \cong (E,a_1, \dots, a_{n+1}).$ Therefore $\overline{\phi}$ is surjective. Now suppose $\overline{\phi}(C,p_1, \dots , p_n) = \overline{\phi} (D, q_1, \dots , q_n).$ Then as abstract marked curves $(C,p_1, \dots , p_{n+1})$ is isomorphic to $ (D, q_1, \dots , q_{n+1})$. If the $j$-invariant of $C \cong D$ is not $0$ or $1728$ then there are exactly two automorphisms of $D$ fixing $q_{n+1}$ \cite[Corollary IV.4.7]{Hart77}. It follows that there are exactly two isomorphisms $(C,p_{n+1}) \cong (D,q_{n+1}).$ We also know that any plane cubic has two linear automorphisms fixing any inflection point: Let $g\in SL(3)$ take $(D,q_{n+1})$ to its Weierstrass form $y^2z = x^3 + axz^2 + bz^3$ such that $g \cdot (q_{n+1}) = (0:1:0).$ Then reflection across the line $y=0$ is a linear automorphism of $g(D,q_{n+1})$  given by the matrix $\diag \{1, -1, 1\}$. This shows that the abstract isomorphism $(C,p_1, \dots , p_{n+1}) \cong (D, q_1, \dots , q_{n+1})$ is one of the two isomorphisms sending $(C,p_{n+1})$ to $(D,q_{n+1})$, both of which are projective equivalences. Therefore $(C,p_1, \dots , p_{n+1})$ is projectively equivalent to $ (D, q_1, \dots , q_{n+1})$, so they are the same point of $\mc{C}_{n,d}^{' \circ} \slash SL(3).$ Then $\overline{\phi}$ is injective on the locus of curves whose $j-$invariant is not $0$ or $1728$. Since $\overline{\phi}$ is a morphism to a normal complex variety which is bijective on an open subset, it is a birational morphism \cite{Starrstack}.

Then, because $M_{g,n}$ is normal and $\overline{\phi}$ is a birational surjection with finite fibers, it follows from Zariski's Main Theorem \cite[Corollary 12.88]{Gortz10} that $\overline{\phi}$ is in fact an isomorphism.
\end{proof}

\printbibliography

@article {CamMoonSchaff,
    AUTHOR = {Caminata, Alessio and Moon, Han-Bom and Schaffler, Luca},
     TITLE = {Determinantal varieties from point configurations on
              hypersurfaces},
   JOURNAL = {Int. Math. Res. Not. IMRN},
  FJOURNAL = {International Mathematics Research Notices. IMRN},
      YEAR = {2023},
    NUMBER = {22},
     PAGES = {19743--19772},
}

@book {FultonInt,
    AUTHOR = {Fulton, William},
     TITLE = {Intersection theory},
    SERIES = {Ergebnisse der Mathematik und ihrer Grenzgebiete (3) [Results
              in Mathematics and Related Areas (3)]},
    VOLUME = {2},
 PUBLISHER = {Springer-Verlag, Berlin},
      YEAR = {1984},
}

@book {Muk03,
    AUTHOR = {Mukai, Shigeru},
     TITLE = {An introduction to invariants and moduli},
    SERIES = {Cambridge Studies in Advanced Mathematics},
    VOLUME = {81},
   EDITION = {Japanese},
 PUBLISHER = {Cambridge University Press, Cambridge},
      YEAR = {2003},
}

@book{Dol03,
  title={Lectures on Invariant Theory},
  author={Dolgachev, Igor},
  volume={},
  year={2003},
  publisher={Cambridge University Press}
}

@book{Dol12,
  title={Classical Algebraic Geometry, A Modern View},
  author={Dolgachev, Igor},
  volume={},
  year={2012},
  publisher={Cambridge University Press}
}

@article{DH98,
title={Variation of geometric invariant theory quotients},
author={Dolgachev, Igor and Hu, Yi},
volume={Publications Mathematiques, no. 87, 5-56},
year={1999},
publisher={IHES}
}

@article {Laz06,
    AUTHOR = {Laza, Radu},
     TITLE = {Deformations of singularities and variation of {GIT}
              quotients},
   JOURNAL = {Trans. Amer. Math. Soc.},
  FJOURNAL = {Transactions of the American Mathematical Society},
    VOLUME = {361},
      YEAR = {2009},
    NUMBER = {4},
     PAGES = {2109--2161},
}

@book {Mum82,
    AUTHOR = {Mumford, D. and Fogarty, J. and Kirwan, F.},
     TITLE = {Geometric invariant theory},
    SERIES = {Ergebnisse der Mathematik und ihrer Grenzgebiete (2)},
    VOLUME = {34},
   EDITION = {Third},
 PUBLISHER = {Springer-Verlag, Berlin},
      YEAR = {1994},
}

@book {Hart77,
    AUTHOR = {Hartshorne, Robin},
     TITLE = {Algebraic geometry},
    SERIES = {Graduate Texts in Mathematics},
    VOLUME = {No. 52},
 PUBLISHER = {Springer-Verlag, New York-Heidelberg},
      YEAR = {1977},
}

@MISC {Starrstack,
    TITLE = {Pushing-forward morphisms},
    AUTHOR = {Jason Starr},
    HOWPUBLISHED = {MathOverflow},
    NOTE = {URL:https://mathoverflow.net/q/225165 (version: 2015-12-03)},
    YEAR = {2015}
}

@book {Gortz10,
    AUTHOR = {G\"ortz, Ulrich and Wedhorn, Torsten},
     TITLE = {Algebraic geometry {I}. {S}chemes---with examples and
              exercises},
    SERIES = {Springer Studium Mathematik---Master},
   EDITION = {Second},
 PUBLISHER = {Springer Spektrum, Wiesbaden},
      YEAR = {2020},
}

@book {LazaThesis,
    AUTHOR = {Laza, Radu-Mihai},
     TITLE = {Deformations of singularities and variations of {GIT}
              quotients},
      NOTE = {Thesis (Ph.D.)--Columbia University},
 PUBLISHER = {ProQuest LLC, Ann Arbor, MI},
      YEAR = {2006},
     PAGES = {286},
      ISBN = {978-0542-64169-5},
}

@book {LazarsPos,
    AUTHOR = {Lazarsfeld, Robert},
     TITLE = {Positivity in algebraic geometry. {I}},
    SERIES = {Ergebnisse der Mathematik und ihrer Grenzgebiete. 3. Folge. A
              Series of Modern Surveys in Mathematics},
    VOLUME = {48},
      NOTE = {Classical setting: line bundles and linear series},
 PUBLISHER = {Springer-Verlag, Berlin},
      YEAR = {2004},
}

@article {hassett2002modulispacesweightedpointed,
    AUTHOR = {Hassett, Brendan},
     TITLE = {Moduli spaces of weighted pointed stable curves},
   JOURNAL = {Adv. Math.},
  FJOURNAL = {Advances in Mathematics},
    VOLUME = {173},
      YEAR = {2003},
    NUMBER = {2},
     PAGES = {316--352},
}

@article {smyth2009classificationmodularcompactificationsmoduli,
    AUTHOR = {Smyth, David Ishii},
     TITLE = {Towards a classification of modular compactifications of
              {${M}_{g,n}$}},
   JOURNAL = {Invent. Math.},
  FJOURNAL = {Inventiones Mathematicae},
    VOLUME = {192},
      YEAR = {2013},
    NUMBER = {2},
     PAGES = {459--503},
}

@article {Ple99,
    AUTHOR = {du Plessis, A. A. and Wall, C. T. C.},
     TITLE = {Curves in {${\rm P}^2({\bf C})$} with {$1$}-dimensional
              symmetry},
   JOURNAL = {Rev. Mat. Complut.},
  FJOURNAL = {Revista Matem\'atica Complutense},
    VOLUME = {12},
      YEAR = {1999},
    NUMBER = {1},
     PAGES = {117--132},
}

@misc{pandharipande1995geometricinvarianttheorycompactification,
      title={A Geometric Invariant Theory Compactification of {$M_{g,n}$} via the {F}ulton-{M}ac{P}herson Configuration Space}, 
      author={R. Pandharipande},
      eprint={alg-geom/9505022},
      year={1995},
      archivePrefix={arXiv}, 
}

@article{Tha96,
title={Geometric invariant theory and flips},
author={Thaddeus, Michael},
volume={9 no. 3},
year={1996},
Journal={Journal of the American Mathematical Society}
}

@article{Bozlee_2023,
   title={A classification of modular compactifications of the space of pointed elliptic curves by Gorenstein curves},
   volume={17},
   ISSN={},
   url={ },
   DOI={ },
   number={1},
   journal={Algebra \&; Number Theory},
   publisher={Mathematical Sciences Publishers},
   author={Bozlee, Sebastian and Kuo, Bob and Neff, Adrian},
   year={2023},
   month=mar, pages={127–163} }

@article{Jen13,
title={Birational Contractions of $\overline{M}3,1$ and $\overline{M}4,1$},
author={Jensen, David},
volume={Transactions of the American Mathematical Society
Vol. 365, No. 6 , pp. 2863-2879 },
year={2013},
publisher={Amer. Mathmatical Soc.}
}

@article {Giansiracusa_2010,
    AUTHOR = {Giansiracusa, Noah and Simpson, Matthew},
     TITLE = {G{IT} compactifications of {$M_{0,n}$} from conics},
   JOURNAL = {Int. Math. Res. Not. IMRN},
  FJOURNAL = {International Mathematics Research Notices. IMRN},
      YEAR = {2011},
    NUMBER = {14},
     PAGES = {3315--3334},
      ISSN = {},
}

@article {giansiracusa2016gitcompactificationsm0nflips,
    AUTHOR = {Giansiracusa, Noah and Jensen, David and Moon, Han-Bom},
     TITLE = {G{IT} compactifications of {$M_{0,n}$} and flips},
   JOURNAL = {Adv. Math.},
  FJOURNAL = {Advances in Mathematics},
    VOLUME = {248},
      YEAR = {2013},
     PAGES = {242--278},
}

@inproceedings {caporaso2018recursivecombinatorialaspectscompactified,
    AUTHOR = {Caporaso, Lucia},
     TITLE = {Recursive combinatorial aspects of compactified moduli spaces},
 BOOKTITLE = {Proceedings of the {I}nternational {C}ongress of
              {M}athematicians---{R}io de {J}aneiro 2018. {V}ol. {II}.
              {I}nvited lectures},
     PAGES = {635--652},
 PUBLISHER = {World Sci. Publ., Hackensack, NJ},
      YEAR = {2018},
      ISBN = {},
   MRCLASS = {},
  MRNUMBER = {},
MRREVIEWER = {},
}

\end{document}